\documentclass[11pt]{article}

\usepackage{amssymb, amsmath, amsfonts, amsthm, graphics, amscd}
\usepackage[dvipsnames]{xcolor}
\usepackage{authblk}
\usepackage{tikz}
\usepackage{pgfplots}
\pgfplotsset{compat=1.8}
\newcommand{\mathsym}[1]{{}}

\def\R{{\mathbf R}}
\def\Z{{\mathbf Z}}
\def\P{{\mathbf P}}
\def\E{{\mathbf E}}

\def\S{{\mathfrak S}}

\DeclareMathOperator{\st}{\mathsf{st}}
\newcommand{\Ex}[1]{\mathbb{E}\{#1\}}

\newtheorem{theorem}{Theorem}[section]
\newtheorem{definition}{Definition}[section]
\newtheorem{proposition}{Proposition}[section]

\newtheorem{lemma}[theorem]{Lemma}

\begin{document}

\title{ 
The Four Point Permutation Test for\\
Latent Block Structure in Incidence Matrices}

\author{R. W. R. Darling}
\author{Cheyne Homberger}
\affil{National Security Agency, Fort George G. Meade,
MD 20755-6844, USA}
\date{\today}
\maketitle

\begin{abstract}
Transactional data may be represented as a bipartite graph $G:=(L \cup R, E)$,
where $L$ denotes agents, $R$ denotes objects visible to many agents, and an
edge in $E$ denotes an interaction between an agent and an object.
Unsupervised learning seeks to detect block structures in the adjacency matrix
$Z$ between $L$ and $R$, thus grouping together sets of agents with similar
object interactions. New results on quasirandom permutations suggest a
non-parametric \textbf{four point test} to measure the amount of block
structure in $G$, with respect to vertex orderings on $L$ and $R$.  Take
disjoint 4-edge random samples, order these four edges by left endpoint, and
count the relative frequencies of the $4!$ possible orderings of the right
endpoint.  When these orderings are equiprobable, the edge set $E$ corresponds
to a quasirandom permutation $\pi$ of $|E|$ symbols.  Total variation distance
of the relative frequency vector away from the uniform distribution on 24
permutations measures the amount of block structure.  Such a test statistic,
based on $\lfloor |E|/4 \rfloor$ samples, is computable in $O(|E|/p)$ time on
$p$ processors. Possibly block structure may be enhanced by precomputing
\textbf{natural orders} on $L$ and $R$, related to the second eigenvector of
graph Laplacians.  In practice this takes $O(d |E|)$ time, where $d$ is the
graph diameter.  Five open problems are described.
\end{abstract}

{\small

\textbf{MSC class: }62H20
}

\section{Introduction}

\subsection{A typical use case - Amazon product reviews}
Amazon Reviews data sets are available for many product categories. 
Sizes of five of these data sets are shown in Table \ref{t:amazon}, located in Section \ref{s;scaling}.
Fix a category, say books,
and consider a bipartite graph $G:=(L \cup R, E)$, where $L$ denotes reviewers,
$R$ denotes books, and an edge corresponds to existence of a review of a specific book
by a specific reviewer. 
Block structure in such a graph corresponds to a clustering of some set of reviewers around some
(unstated) type of books. The four point test developed in this paper quantifies the amount
of block structure on a scale from 0 to 1. For example, in the original ordering,
books received a 0.3 score while digital music received a 0.6, implying that reviewers of
digital music are more bound to their music genres than are book reviewers to their type of book.

\subsection{Notation for incidence data} \label{s:notation}

Association mining treats an incidence matrix, represented as an
 undirected \textbf{bipartite graph} $G:=(L \cup R, E)$ with ordered
left vertices $L:=\{u_1, \ldots, u_n\}$, degrees $\mathbf{w}:=(w_1, \ldots, w_n)$; and
ordered right vertices $R:=\{v_1, \ldots, v_m\}$, degrees $\mathbf{d}:=(d_1, \ldots, d_m)$.
There are $|E|=N$ incidences of form $(u_i, v_j) \in L \times R$, also written $u_i \sim v_j$,
where
\begin{equation} \label{e:degreesums}
\sum_i w_i = N = \sum_j d_j.
\end{equation}
At a finer level of detail \cite{sta}, the \textbf{joint degree matrix} $(N_{w,d})$ of $G$
is the integer matrix whose $N_{w,d}$ entry counts the number of edges $e \in E$
whose left endpoint has degree $w$, and whose right endpoint has degree $d$:
\begin{equation} \label{e:jointdegmat}
N_{w,d}:=|\{(u_i,v_j) \in E: w_i = w, d_j = d\}|.
\end{equation}

Alternatively, we may view the data as:
\begin{enumerate}
\item A \textbf{binary contingency table}, i.e. a 0-1 matrix $Z:=(Z_{i,j})$
 with given row sums $\mathbf{w}:=(w_1, \ldots, w_n)$
 and column sums $\mathbf{d}:=(d_1, \ldots, d_m)$. Here
\begin{equation} \label{e:incidences}
Z_{i,j}:=1_{u_i \sim v_j}.
\end{equation}

\item A \textbf{hypergraph} $(R, \mathcal{E})$ with given vertex degrees and hyperedge weights.
Here $\mathcal{E} = \{e_1, e_2, \ldots, e_n\}$ is in bijection with the left nodes $L$, and
\[
e_i:=\{v_j \in R: u_i \sim v_j \}.
\]
\end{enumerate}

\subsection{What is meant by absence of block structure?} \label{s:blockfree}
Classical studies of association in small contingency tables, summarized in Agresti \cite{agr},
focus on tests for statistical independence of rows in a small, dense random matrix. 
For reasons discussed in Appendix \ref{s:lrs2models}, such notions are entirely unsuited to
the discovery of block structure in large, sparse binary contingency tables.
Instead we follow a non-parametric approach suggested by the quasirandomness literature 
\cite{kra, len, len2, sha}. 

Repeat the following experiment many times:
\textit{pick $s$ edges uniformly at random, sort according to left end point, and
test whether all orderings of right end points are equally frequent.}
This approach is too crude to detect statistical dependence between a specific pair of rows
in a large matrix, but is able to detect block structure, as we shall see in
Section \ref{s:2blocks}. Moreover we will not need to consider arbitrarily large $s$;
the choice $s=4$ suffices.

The first tool we shall introduce is a method of converting a sample of
$s$ edges into an element of
$\S_s$,  where $\S_s$ denotes
the set of all permutations of length $s$.
All the definitions of this section
extend to bipartite multigraphs, in which a sample of edges might include two edges
which have the same pair of endpoints.

\begin{definition} \label{d:stoperm}
Sample $s \geq 2$ distinct edges 
\[
(u_{i(1)}, v_{j(1)}), (u_{i(2)}, v_{j(2)}), \ldots, (u_{i(s)}, v_{j(s)})
\]
in a bipartite (multi)graph $G:=(L \cup R, E)$ where $L$ and $R$ are ordered,
and sort them by left end point, so $i(1) \leq i(2) \leq \cdots \leq i(s)$.
A permutation $\sigma \in \S_s$ \textbf{induced by} the sample
means one that is selected uniformly at random from those with
the property\footnote{
If there are no ties in either sorted list, then there is
a unique $\sigma$ with this property.} :
\[
j(\sigma(1)) \leq j(\sigma(2)) \leq \cdots \leq j(\sigma(s)).
\]
\end{definition}

How could this allow us to test for absence of block structure?
Suppose $2 \leq s < N$ edges are picked uniformly at random, and sorted by left endpoint.
Repeat this many times.
If there is some $J \subset R$ such that right vertices $v_j \in J$ tend 
to appear earlier on such a list than those where $j \in R \setminus J$, then
 the $s!$ possible orderings of right endpoints in the sample are not 
equally likely. In this case lower numbered left vertices would tend to be associated with
right vertices $\{v_j \in J \}$. 
We shall study an explicit example in Section \ref{s:2blocks}.

\begin{definition} \label{d:abfos}
A sequence of random bipartite (multi)graphs
$G_k:=(L_k \cup R_k, E_k)$,
where $|L_k| \to \infty$, $|R_k| \to \infty$, $|E_k| \to \infty$
is called \textbf{asymptotically block-free of order $s$} if
the distribution of the permutation induced by
 a uniform random sample of $s$ distinct edges
converges to the uniform distribution on $S_s$, as $k \to \infty$.
If this condition holds for all $s \geq 2$, we call $(G_k)$
asymptotically block-free.
\end{definition}

\textbf{Remark: }
Lemma \ref{l:abffromperm} shows how such a sequence may be constructed.

In a practical situation, we typically have a single large graph, from which
we can draw many samples of size $s$. By partitioning the $N$ edges randomly into
sets of size $s$, we obtain $t:=\lfloor N/s \rfloor$ such samples, each of which
induces one of $s!$ permutations, as in Definition \ref{d:stoperm}. A
null hypothesis $H_0^s$ could now be phrased as: each of the $s!$ possible
outcomes in these $t$ independent multinomial trials has equal probability $1/s!$.
An alternative hypothesis $H_1^s$ is that these $s!$ possible
outcomes are not equally likely. The $\chi^2$ goodness of fit test to the
multinomial, with $s! -1$ degrees of freedom would be a natural choice to test $H_0^s$ versus $H_1^s$.

This procedure is still burdensome: it seems we must repeat for all $s \leq N/2$,
and then decide how to combine the results.
Fortunately it suffices to consider just the case $s = 4$. In other
words, the only hypothesis we need to test is $H_0^4$, which means that, for the 
given orderings of left and right vertices, the graph is block-free of order four.

\subsection{Permutations induced by samples of four edges suffice}

The main result of our paper is:
\begin{theorem}\label{t:abf4}
If a sequence of random bipartite (multi)graphs, as in Definition \ref{d:abfos}, is 
asymptotically block-free of order four, then it is asymptotically block-free of order
$s$ for all $s \geq 2$.
\end{theorem}

The proof, which will be given later, comes from combining a combinatorics 
result of Kr{\'a}l$'$ \& Pikhurko \cite{kra}
concerning quasirandom permutations,
with a construction which maps a bipartite (multi)graph with $N$ edges to a permutation on
$N$ symbols.

\section{Patterns in Permutations}

We'll borrow some machinery from the study of permutation patterns. Use the \emph{one-line
notation} for permutations: identify $\pi \in \S_n$ with the sequence
$\pi(1) \pi(2) \ldots \pi(n)$. For example, 
\[ \S_3 = \{123, 132, 213, 231, 312, 321\}. \]

We start with two definitions:

\begin{definition}
  For any positive integer $n$ and any sequence of distinct real numbers $r = r_1,
  r_2, \ldots r_n$, we define the \emph{standardization} of $r$, denoted
  $\st(r)$, to be the unique permutation $\pi \in \S_n$ such that, for all $1
  \leq i,j \leq n$, 
  \[ r_i < r_j \quad \text{if and only if} \quad \pi(i) < \pi(j). \]
\end{definition}

\begin{definition}
  Let $\pi$ and $\sigma$ be two permutations. We say that \emph{$\sigma$ is
  contained as a pattern in $\pi$} if $\sigma$ is the standardization of some
  subsequence of $\pi$. We denote this by $\sigma \prec \pi$, and say that $\pi$ has a $\sigma$-pattern.
\end{definition}

For example, the permutation $213$ is contained as a pattern in the permutation
$531426$, since the standardization of the second, third, and final entry is
equal to $213$.  The set of all permutations equipped with this containment
order forms an infinite graded poset. The study of permutations patterns has a
rich history: for a survey of results in the area, see B\'ona~\cite{bonabook}. 

We'll apply some recent results concerning pattern \emph{counts} in random
permutations. For any permutations $\pi$ and $\sigma$, define $\nu_\sigma(\pi)$
to be the number of times $\sigma$ appears as a pattern within $\pi$. Note that
$\nu_\sigma$ is a function from the set of all permutations to the non-negative
integers. A result of Br\"and\'en and Claesson~\cite{branden} shows that every
permutation statistic can be expressed as a linear combination of pattern-counting functions. 

We now define a class of random variables based on pattern counts. For a
positive integer $n$ and a permutation $\sigma$, let $X_{n, \sigma}$ be 
the probability that in a random $n$-permutation, any fixed $k$-subset
of entries forms a $\sigma$-pattern.  It follows by linearity of expectation
that if $\sigma$ is any permutation of length $k$, for any fixed $k$-subset of entries we have 
\[ \Ex{X_{n, \sigma}} = \frac{1}{k!}. \]

B\'ona~\cite{bona2} showed
that, for any permutation $\sigma$, $X_{n, \sigma}$ is asymptotically normally
distributed as $n \rightarrow \infty$. Note, however, that while the mean of
this variable depends only on the length of $\sigma$, the variance depends on
the specific choice of pattern. Janson, Nakamura, and Zeilberger~\cite{jan3}
extended this result to show that for any two permutations $\sigma$ and $\tau$,
the random variables $X_{n, \sigma}$ and $X_{n, \tau}$ are jointly
asymptotically normally distributed as $n \rightarrow \infty$. 

Pattern occurrences are far from independent: for example, the number of $12$-patterns is clearly 
negatively correlated with the number of $21$ occurrences.
Figure~\ref{fig:pattern_correlations} shows the correlation of all length-4
patterns across the set of all permutations of length 8. 

\begin{figure}[ht]
  \centering
  \includegraphics[width=4in]{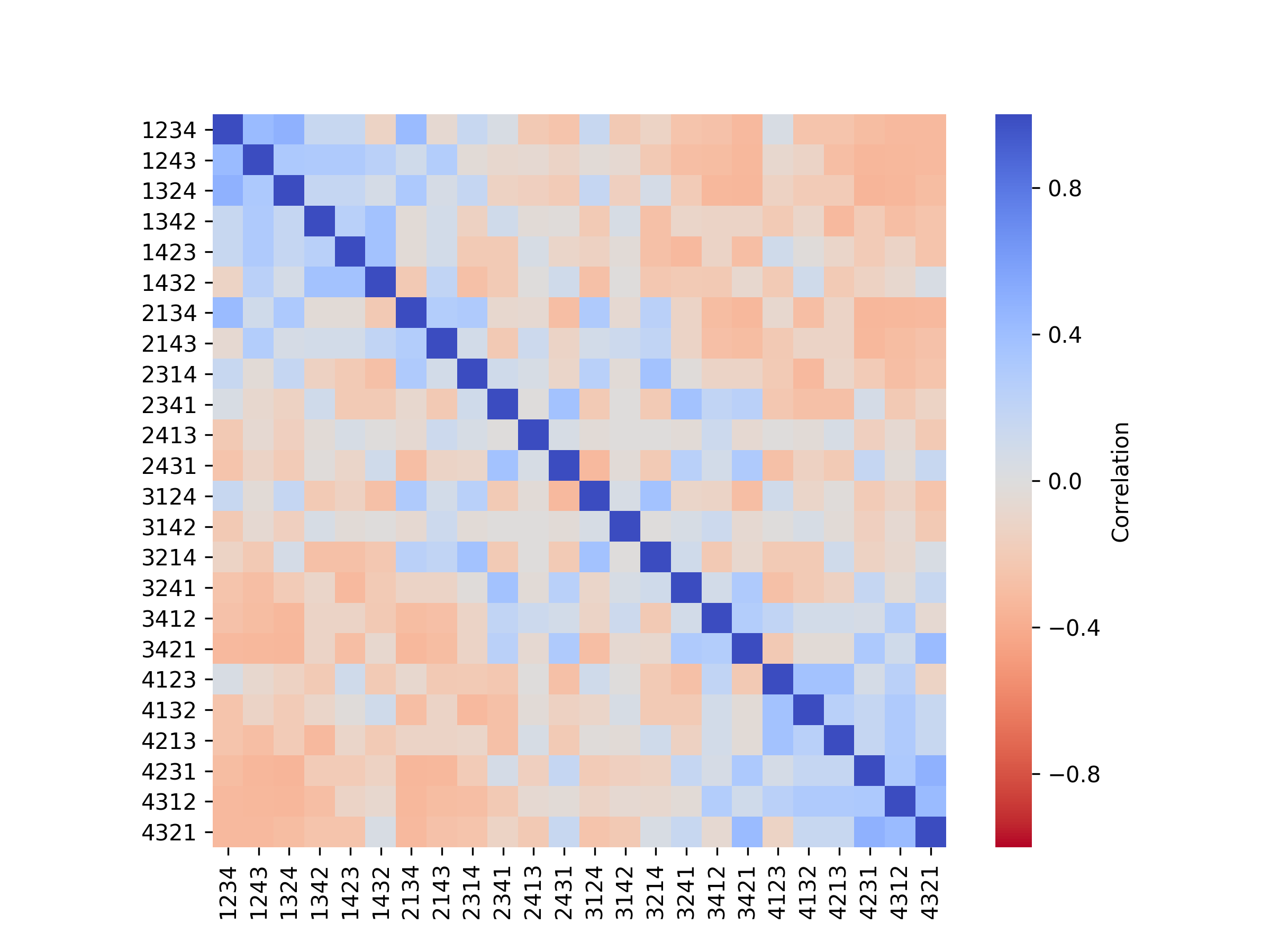}
  \caption{Correlation between length-4 pattern counts within all
  8-permutations.}
  \label{fig:pattern_correlations}
\end{figure}

\subsection{Permutation Classes and Structure}

Permutation patterns provide a framework for analyzing the structure of
permutations. The connection between patterns and block structure is most
clearly seen in the case of the \emph{separable} permutations, which we define
here after a few preliminary definitions.

For permutations $\sigma$ and $\tau$ of length $n$ and $m$, define
their \emph{direct}- and \emph{skew}-sum, denoted $\sigma \oplus \tau$ and
$\sigma \ominus \tau$ to be $(n+m)$-permutations as follows:
\begin{align*}
  (\sigma \oplus \tau)(i) &= 
      \begin{cases}
      \sigma(i) & 1 \leq i \leq n  \\
      \tau(i-n) + n &  n < i \leq n+m 
      \end{cases} \\
  (\sigma \ominus \tau)(i) &= 
    \begin{cases}
      \sigma(i) + m & 1 \leq i \leq n \\
      \tau(i-n)  & n < i \leq n+m
      \end{cases} 
    .
\end{align*}

These operations are more intuitively understood graphically in terms of the
plots of $\sigma$ and $\tau$. The direct sum places the plot of $\sigma$ below
and to the left of that of $\tau$, while the skew sum places it above and to the
left. See Figure~\ref{fig:direct_skew_sum}. 

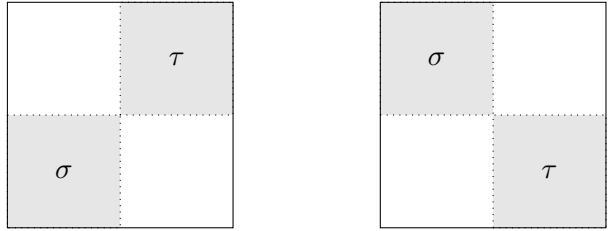
\begin{figure}[ht]
  \centering
      \begin{tikzpicture}[scale=.25]
        \draw[dotted, fill = black!10] 
              (0,6) -- (6,6) -- (6,0) -- (0,0) -- cycle;
        \draw[dotted, fill = black!10] 
              (6,6) -- (6,12) -- (12,12) -- (12,6) -- cycle;
        \draw (0,0) -- (12,0) -- (12,12) -- (0,12) -- cycle;
        \node at (3,3) {$\sigma$};
        \node at (9,9) {$\tau$};
      \end{tikzpicture}
      \hspace{4pc}
      \begin{tikzpicture}[scale=.25]
        \draw[dotted, fill = black!10] 
              (0,6) -- (6,6) -- (6,12) -- (0,12) -- cycle;
        \draw[dotted, fill = black!10] 
              (6,0) -- (6,6) -- (12,6) -- (12,0) -- cycle;
        \draw (0,0) -- (12,0) -- (12,12) -- (0,12) -- cycle;
        \node at (3,9) {$\sigma$};
        \node at (9,3) {$\tau$};
      \end{tikzpicture}

  \caption{The plots of the permutations $\sigma \oplus \tau$ and $\sigma
  \ominus \tau$. } \label{fig:direct_skew_sum}
\end{figure}

A permutation is said to be \emph{sum- (resp., skew) indecomposable} if it
cannot be written as the direct (resp., skew) sum of two permutations. A
\emph{decomposable} permutation is one which can be written as either a direct
or a skew sum in some way.

A \emph{separable} permutation is one which can be decomposed as sums of the
trivial permutation of length 1. For example, the permutation $\pi = 215643798$ is
separable, since
\[ \pi = \Big(1 \ominus 1\Big) \oplus \Big( (1 \oplus 1) \ominus 
       1 \ominus 1 \Big) \oplus 1 \oplus \Big(1 \ominus 1\Big).\]
Equivalently, a separable permutation is one which is \emph{recursively
decomposable}: it can be decomposed into blocks which themselves can be
decomposed into blocks, which themselves can be decomposed into blocks, etc. See
Figure~\ref{fig:separable} for the decomposition of the plot of $\pi$. 

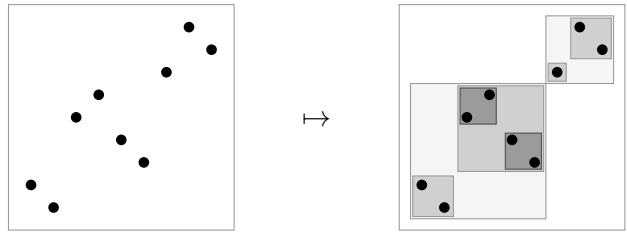
\begin{figure}[ht]
  \centering
  \begin{tikzpicture}[scale=.3]
      \draw[draw=black!40] (0,0) rectangle (10,10);
      \foreach [count=\x] \y in {2,1,5,6,4,3,7,9,8}{
      \node[circle,inner sep=.5mm, fill=black] at (\x,\y) {};
      }
  \end{tikzpicture}
  \qquad
  \raisebox{3.3pc}{$\mapsto$}
  \qquad
  \begin{tikzpicture}[scale=.3]
      \draw[draw=black!40] (0,0) rectangle (10,10);

      \draw[draw=black!40, fill opacity=.4, fill=black!10] 
        (.5,.5) rectangle (6.5,6.5);
      \draw[draw=black!40, fill opacity=.4, fill=black!10] 
        (6.5,6.5) rectangle (9.5,9.5);

      \draw[draw=black!40, fill opacity=.4, fill=black!40] 
        (.6,.6) rectangle (2.4,2.4);
      \draw[draw=black!40, fill opacity=.4, fill=black!40] 
        (2.6,2.6) rectangle (6.4,6.4);
      \draw[draw=black!40, fill opacity=.4, fill=black!40] 
        (6.6,6.6) rectangle (7.4,7.4);
      \draw[draw=black!40, fill opacity=.4, fill=black!40] 
        (7.6,7.6) rectangle (9.4,9.4);

      \draw[draw=black!70, fill opacity=.4, fill=black!70] 
        (2.7, 4.7) rectangle (4.3, 6.3);
      \draw[draw=black!70, fill opacity=.4, fill=black!70] 
        (4.7, 4.3) rectangle (6.3, 2.7);

      \foreach [count=\x] \y in {2,1,5,6,4,3,7,9,8}{
      \node[circle,inner sep=.5mm, fill=black] at (\x,\y) {};
      }
  \end{tikzpicture}

  \caption{The recursive block-decomposition of the plot of the permutation
  $\pi=215643798$.}
  \label{fig:separable}
\end{figure}

The set of separable permutations also has another important characterization: a
permutation is separable if and only if it does not contain either of the
patterns 2413 or 3142. Thus, a randomly chosen permutation which has this
recursive block structure must have no occurrences of these two patterns. In
fact, the block structure has an impact on other patterns as well. Using results
from Albert, Homberger, and Pantone~\cite{albert} or from Bassino,
et.al.~\cite{bouvel}, we can calculate the expected number of occurrences in a
random separable permutation of length $n$. 
Let $q'_{n, \sigma}$ denote the probability that a randomly chosen 4-subset of a
randomly chosen separable permutation of length $n$ forms a $\sigma$-pattern. 
Consider a random separable permutation of length $n$, and let $q'_\sigma = \lim_{n \rightarrow \infty} q'_{n, \sigma}$ be the
asymptotic probability that a randomly chosen $4$-subset forms a
$\sigma$-pattern.
We have:
\begin{equation} \label{eqn:separable_expectations}
    q'_{\sigma_1} = 1/8, \qquad q'_{\sigma_2} = 1/20, 
\qquad q'_{\sigma_3} = 1/40, \qquad q'_{\sigma_4} = 0, 
\end{equation} 
where
{ \small \begin{align*}
  \sigma_1 & \in \{1234, 4321\} \\
  \sigma_2 & \in \{1243, 1324, 1432, 2134, 2341, 3214, 3421, 4123, 4231, 4312 \}\\
  \sigma_3 & \in \{1342, 1423, 2143, 2314, 2431, 3124, 3241, 3412, 4132, 4213 \}\\
  \sigma_4 & \in \{2413, 3142 \}.
\end{align*}}

Recall that in the set of all permutations, all patterns are equally likely.
Intuitively, this shows that a block structure within a permutation affects the
number of occurrences of patterns. 

\subsection{Non-Overlapping Patterns}

We consider a related problem: counting pattern occurrences within a single, large permutation. If we were to count all patterns, we would be counting individual entries many times: a single occurrence of the pattern 12345, for example, would lead to 5 separate occurrences of the pattern 1234. Instead, we'll count non-overlapping patterns. Let $\pi$ be a permutation of length $n$, let $k = \lfloor n/4 \rfloor$, and let $S:= \{S_i\}_{i=1}^k$ be a family of randomly chosen disjoint subsets of $[n]$, each of size 4. We consider the multiset of $k$ patterns of length 4 located at each of these sets of indices. 

It follows by linearity of expectation that if $\pi$ is chosen uniformly at random and $\sigma$ is any pattern of length 4, then the expected number of times that $\sigma$ appears across the index sets $S$ is equal to $1/24$. Lemma \ref{l:disjointperms} provides a tool
to simplify distribution of test statistics in later sections.

\begin{lemma} \label{l:disjointperms}
Suppose $\pi$ is a uniform random permutation of length $n$, and let $\{S_i\}_{i=1}^k$ be a family of randomly chosen disjoint subsets of $[n]$, each of size 4.
For any $i \neq j$, the patterns formed at indices $S_i, S_j$ are statistically independent.
\end{lemma}

\begin{proof}
We can fix all of the indices outside of $S_i \cup S_j$ and permute the entries of these sets in any way. This shows that there are precisely the same number of permutations $\pi$ having any specified pair of patterns at $S_i, S_j$. 
\end{proof}
The same, however, is not true in the separable case, since permuting these entries may lead to a forbidden pattern.

For example, Figure~\ref{fig:rand_separable} shows a uniformly randomly chosen separable permutation of length $10,379$. Figure~\ref{fig:compare_sep_to_random} shows a comparison in the non-overlapping pattern counts between this permutation and those of a permutation chosen uniformly at random from the set of all permutations of the same length.

\begin{figure}[ht]
    \centering
    \includegraphics[width=3in]{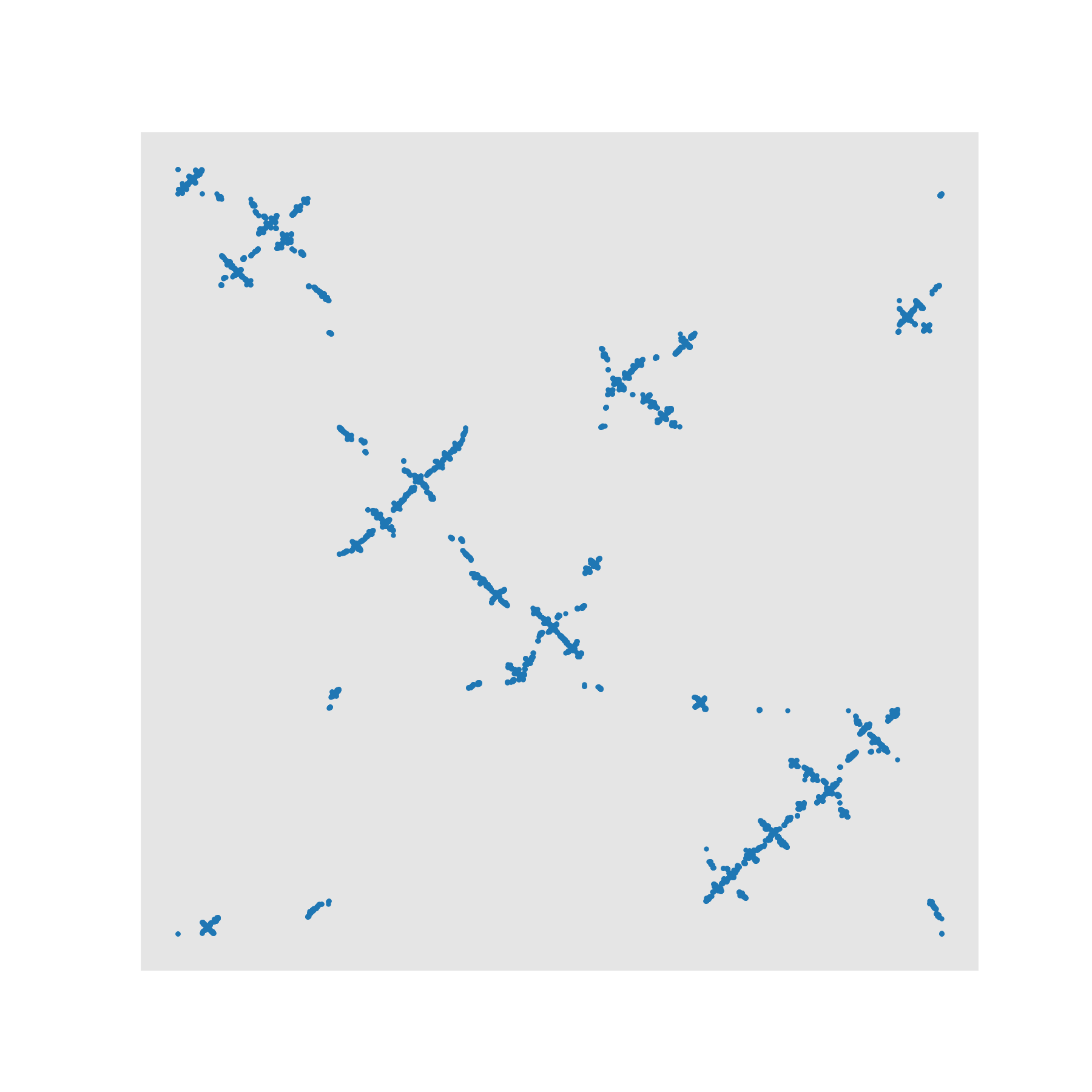}
    \caption{A random separable permutation of length $10,379$.}
    \label{fig:rand_separable}
\end{figure}

\begin{figure}[ht]
    \centering
    \includegraphics[width=4in]{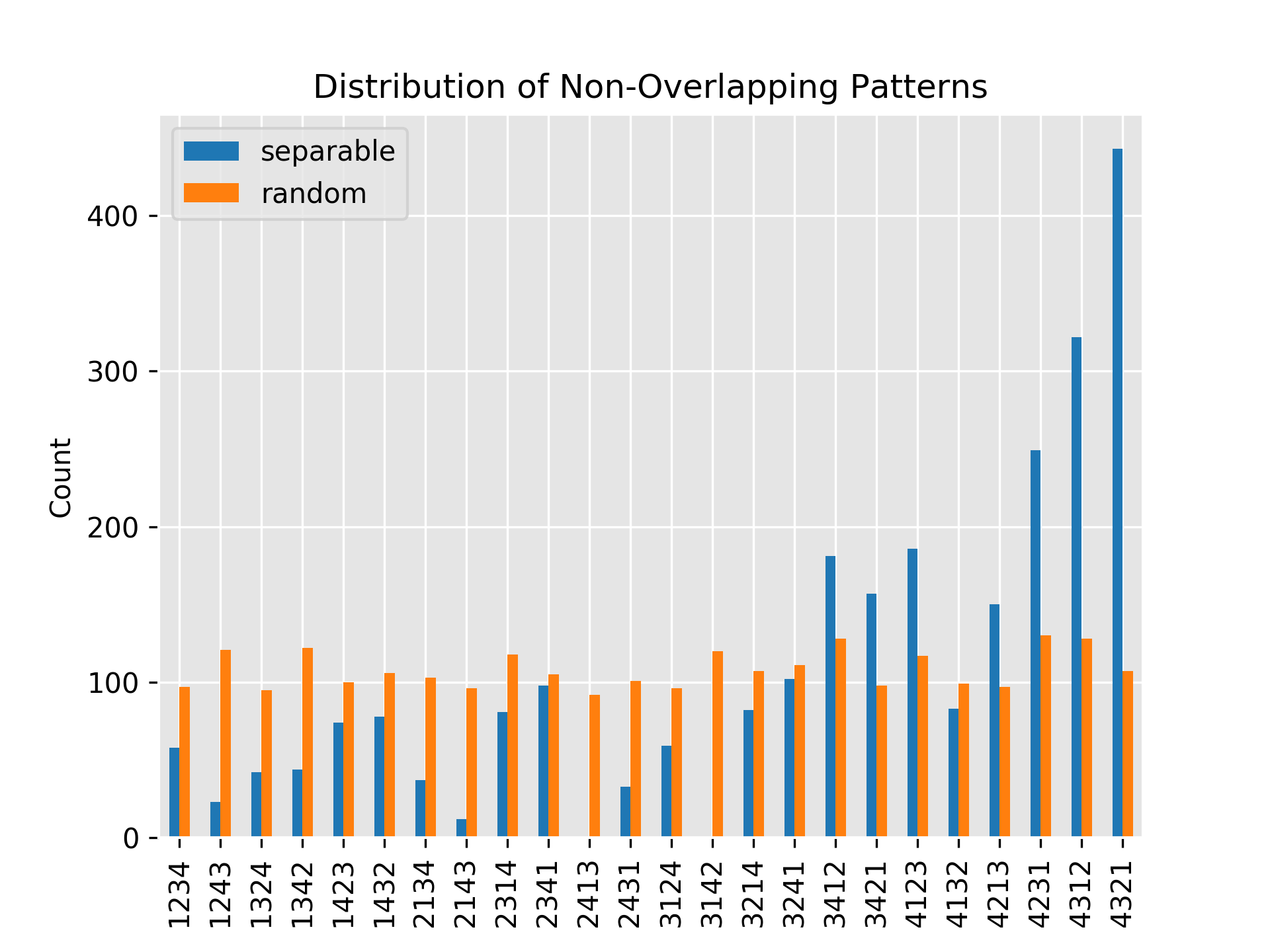}
    \caption{Comparison of non-overlapping pattern counts in a randomly chosen permutation and a randomly chosen separable permutation.}
    \label{fig:compare_sep_to_random}
\end{figure}

\section{Application: measuring block structure}
 
\subsection{4-permutations and Block Structure} 
We are given a sparse bipartite (multi)graph with $n$ edges, with total orders on the left vertices
and on the right vertices.

We define a permutation $\pi$ of length $n$ based on the graph as follows: sort
the sequence of edges according to their left endpoint, and let $a = a_1, a_2, \dots
a_n$ be the sequence of right endpoints. The values of this sequence are not
necessarily distinct, but we can create a distinct sequence $a'$ by introducing
some small random noise: let  $a'_i = a_i + X_i$, where
$\{X_i\}_{i=1}^n$ is an i.i.d. sequence distributed uniformly on $[0,1]$. 
Let $\pi$ be the standardization of $a'$. 

Now, take four elements from $\pi$ at random, and consider their
standardization.

The \textbf{null hypothesis} $H_0^4$ says: each of the $24$ possible
outcomes is equally likely.

The \textbf{alternative hypothesis} $H_1^4$ says: some patterns of length 4 are
more likely than others. 

Here is how we propose to perform the test of $H_0^4$ versus $H_1^4$ in $O(|E|)$
time, or indeed $O(|E|/p)$ time if Steps 2 and 3 of Section \ref{s:4pttest} are
distributed among $p$ processors.

\subsection{Lehmer codes: a convenient tool}
In practical computation, the ordering of right endpoints may be represented by 
the \textbf{Lehmer code}\footnote{
Suggested by Ryan Kaliszewski, personal communication}
which maps the sequence $(v_1, v_2, v_3, v_4)$ to $(L_1, L_2, L_3, L_4)$, where
\begin{equation} \label{e:Lehmer}
L_i = \# \{j > i: v_i > v_j\} \in \{0, 1, \ldots, 4-i\}.
\end{equation}
For example $(141, 817, 96, 108)$ has Lehmer code $(2, 2, 0, 0)$. Next the mapping
\begin{equation} \label{e:Lehmer2perm}
(L_1, L_2, L_3) \rightarrow 6 L_1 + 2 L_2 + L_3
\end{equation}
is a bijection from $\{0,1,2,3\} \times \{0,1,2\} \times \{0,1\}$ to $\{0,1,\ldots, 23\}$,
bearing in mind that $L_4 = 0$.

\subsection{Four point test: computational steps}\label{s:4pttest}
Recall that $L$ and $R$ are ordered sets of vertices, inducing two partial orders
on the set $E$ of $N$ edges, namely the partial order by left endpoint, and
the partial order by right endpoint, respectively.

\begin{enumerate}\setcounter{enumi}{-1}

\item
For tie-breaking purposes, select independently, and uniformly at random, 
total orders $\prec_L$ and $\prec_R$ on $E$ among the linear
extensions of the partial orders induced by those on $L$ and $R$, respectively.
For example, if $L$ and $R$ are sets of integers, this can be achieved
by jittering each $e:=(u_e,v_e) \in E \subset \Z^2$ to $(u_e + \eta_e,v_e + \eta'_e)$,
where $(\eta_e,\eta'_e)_{e \in E}$ are pairs of independent Uniform$(-b, b)$ random variables,
for $b < 1/2$.

\item
Draw $\lfloor N/4 \rfloor$ samples\footnote{
In practice, order $E$ randomly, then partition it into blocks of length four, discarding
any remainder.
} of size four from $E$, uniformly and without replacement.
Thus no edge is sampled more than once.

\item
Order each block of size four, say $e_i:=(u_i, v_i)$, $1 \leq i \leq 4$,
by $\prec_L$:
\begin{equation} \label{e:4-order}
e_1 \prec_L e_2 \prec_L e_3 \prec_L e_4,
\end{equation}
so $u_1 < u_2 < u_3 < u_4$ in $L$ if all these left vertices are distinct.
Compute the standardization associated with the 
ordering of $(e_i)_{1 \leq i \leq 4}$ under $\prec_R$, which coincides with the ordering
of $(v_1, v_2, v_3, v_4)$ in $R$, if all these right vertices are distinct.
For example if the block of four is 
\[
\{(-310, 96), (-477, 817), (-621, 141), (-65, 108)\},
\]
sorting by left vertex gives 
\[\{(-621, 141), (-477, 817), (-310, 96), (-65, 108)\},
\]
and the Lehmer code for $(141, 817, 96, 108)$ is $(2, 2, 0, 0)$.

\item
These $\lfloor N/4 \rfloor$ samples yield a vector $\mathbf{X}:=(X_0, X_1, \ldots, X_{23})$
counting the frequencies of each pattern.

Under the null hypothesis $H_0^4$, 
Lemma \ref{l:disjointperms} ensures that
$\mathbf{X}$
stores $\lfloor N/4 \rfloor$ independent 
multinomial$(\frac{1}{24}, \ldots, \frac{1}{24})$ trials.

\item
\begin{itemize}
\item[(a)]
Suppose we wish to test the null hypothesis $H_0^4$, i.e.
the graph is block-free of order four, for the 
given orderings of left and right vertices.
Perform a $\chi^2$ goodness of fit test of $\mathbf{X}$ with respect to the 
multinomial$(\lfloor N/4 \rfloor, \mathbf{p})$ distribution.

The expected value of each
$X_i$ under the null hypothesis $H_0^4$ is $\theta:= \lfloor N/4 \rfloor /24$. Compare the
\textbf{four point chi-squared statistic} (\textbf{4PT}-$\chi^2$ for short)
\begin{equation} \label{e:4ptteststat}
T_4:= \sum_{i=0}^{23} \frac{(X_i - \theta)^2}{\theta}
\end{equation}
to the upper tail of the $\chi^2(23)$ distribution.

\item[(b)]
Suppose $H_0^4$ has been rejected, and
we seek a scale-free measure of how much block structure the graph has, 
with respect to the given orderings of left and right vertices.
We propose to use the \textbf{total variation distance}, or \textbf{4PT-TV},  between the empirical probability
measure which assigns mass $X_i / \lfloor N/4 \rfloor$ to Lehmer code $i$,
and the uniform measure on the 24 Lehmer codes, namely

\begin{equation} \label{e:imbalance}
D_4:= \frac{1}{2 \lfloor N/4 \rfloor} \sum_{i=0}^{23} |X_i - \theta| \in [0,1].
\end{equation}

\end{itemize}

\end{enumerate}

\subsection{Basic example: bipartite graph with two blocks} \label{s:2blocks}

\begin{figure}
\caption{ \textit{
Bipartite graph whose incidence matrix decomposes into two blocks.}
} \label{f:2blocks}
\begin{center}
\scalebox{0.4}{\includegraphics{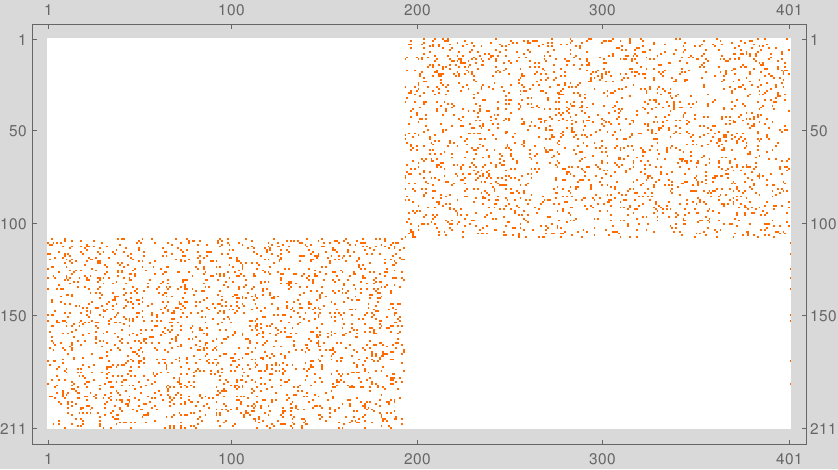}} 
\end{center}
\end{figure}

Figure \ref{f:2blocks} shows an example where all the incidences in the bipartite graph
fall either in $A \times B$ or in $A^c \times B^c$, for some $A \subset L$ and $B \subset R$.
Suppose proportion $\alpha$ of incidences fall into $A \times B$, and 
proportion $1-\alpha$ fall into $A^c \times B^c$.
As in Figure \ref{f:2blocks}, suppose vertices in $A$ are listed before those in $A^c$, and those in
$B$ are listed before those in $B^c$. Call this the
\textbf{ordered two block model}.

Suppose four incidences $(u_1, v_1)$, $(u_2, v_2)$, $(u_3, v_3)$, $(u_4, v_4)$
are selected uniformly at random, labelled so that
\[
(u_1, v_1) \prec_L (u_2, v_2) \prec_L (u_3, v_3) \prec_L (u_4, v_4)
\]
as in (\ref{e:4-order}).
Let $Y$ denote the number of these incidences which belong to the $A \times B$ block.
Then $Y\sim$ Binomial$(4, \alpha)$. For example, when $Y=2$, the ordering
(\ref{e:4-order}) implies that
\[
(u_1, v_1), (u_2, v_2) \in A \times B; \quad (u_3, v_3), (u_4, v_4) \in A^c \times B^c.
\]
Hence out of the 24 permutations, the only possible ones when $Y=2$ are those in the set
\[
\Pi_2:=\{1 2 3 4, 1 2 4 3, 2 1 3 4, 2 1 4 3\}.
\]
Likewise $\Pi_1$ consists of permutations where 1 is in the first place, 
$\Pi_3$ consists of permutations where 4 is in the last place, while $\Pi_0 = \S_4 = \Pi_4$.
From this reasoning, we obtain the simple lemma:

\begin{lemma}
In the ordered two block model, where a proportion $\alpha$ of incidences fall into the 
$A \times B$ block, and 
proportion $1-\alpha$ fall into the $A^c \times B^c$
block, the relative frequency of permutation $\pi \in \S_4$ is
\begin{equation} \label{e:permfreq}
h(\pi, \alpha):=\frac{f(\pi, \alpha)} {\sum_{\pi' \in \S_4} f(\pi', \alpha)};
\quad
f(\pi, \alpha): = \sum_{k=0}^4 \P[Y=k] 1_{ \{\pi \in \Pi_k  \} },
\end{equation}
where $Y\sim$ Binomial$(4, \alpha)$, and $\Pi_k$ is the set of permutations which are possible
under the constraint that the first $k$ of $u_1, u_2, u_3, u_4$ belong to $A$.
\end{lemma}

These relative frequencies are displayed in Figure \ref{f:relfreq}
as a function of $\alpha$. This number
of curves is less than 24 because there exist different choices of $\pi \in\S_4$ for which the
functions $\alpha \to f(\pi, \alpha)$ coincide.

\begin{figure}
\caption{\textit{
For $\alpha$ on the horizontal axis, the 
vertical axis shows the
relative frequency $h(\pi, \alpha)$
 of permutation $\pi$ of right vertex orderings, when incidences fall into one of two blocks,
at rates $\alpha$ and $1-\alpha$, respectively. 
The 24 choices of $\pi$ yield only 8 distinct curves.}
} \label{f:relfreq}
\begin{center}
\scalebox{0.4}{\includegraphics{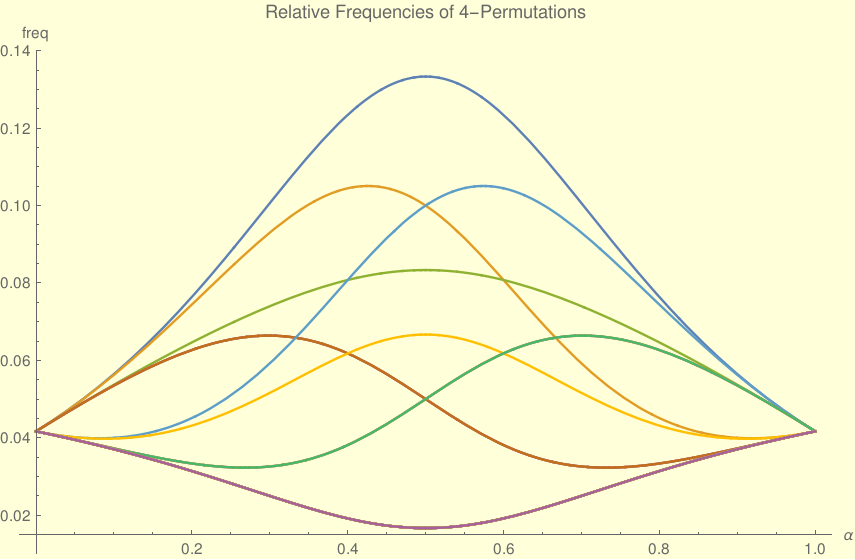}} 
\end{center}

\end{figure}

The permutations $\pi \in \S_4$ for which $f(\pi, 0.5) = 1/8$ (the lowest value) are
those in $\S_4 \setminus \{\Pi_1 \cup \Pi_2 \cup \Pi_3\}$.

Figure \ref{f:relfreq} demonstrates that
when vertex ordering reveals 
block structure in the 
incidence matrix, the relative frequencies of different permutations in $\S_4$ are tilted, just as they for
separable permutations (Figure \ref{fig:compare_sep_to_random}).

\section{How vertex ordering affects the four point test} \label{s:orderimportance}

\subsection{A pair of superficially similar but structurally different matrices}
We shall set up a pair of Bernoulli matrix models, $(Z_{i,j})$ and
$(Z'_{i,j})$, each with $n$ rows and $m$ columns,
whose marginal statistics and likelihood ratio statistics (see Section \ref{s:lrs2models})
are almost indistinguishable, 
but whose structure is entirely different, and then apply the four point test to each.
We will also show how changing the vertex ordering of one of them dramatically changes the
results of the four point test.

\begin{figure}
\caption{\textit{
The left glob represents Mathematica's \cite{wol} attempt at vertex partitioning
 for the bipartite graph
corresponding to the left incidence matrix in Figure \ref{f:2incmats}, an instance of
quasirandom structure as in Section \ref{s:bmm-noblocks}. 
The right pair of globs shows how the same algorithm partitions the bipartite graph
corresponding to the right incidence matrix in Figure \ref{f:2incmats}, an instance of
the two-block model of Section \ref{s:bmm-blocks}. 
}
} \label{f:2graphpartitions}
\begin{center}
\scalebox{0.4}{\includegraphics{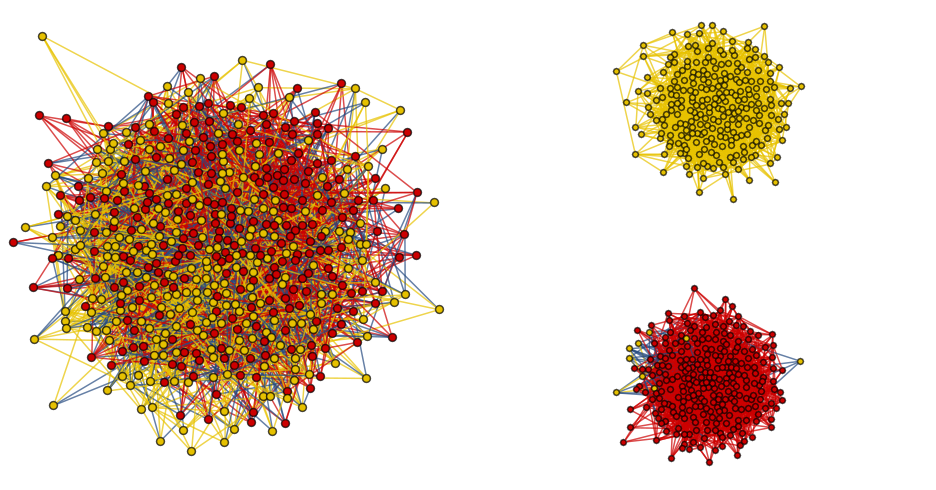} }
\end{center}
\end{figure}

Figures \ref{f:2graphpartitions} and \ref{f:2incmats} 
illustrate, respectively, (1) the structural difference between the two associated
bipartite graphs, one of which decomposes completely into two components, like the
one shown in Figure \ref{f:2blocks}, and (2) the superficial similarity of their
incidence matrices, under suitably randomized vertex orderings.

\begin{figure}
\caption{\textit{
Two $307 \times 211$ binary matrices
are shown. The left ($3116$ entries) has
pseudo-random structure as in Section \ref{s:bmm-noblocks}, and the right 
($3065$ entries) is the two-block model
of Section \ref{s:bmm-blocks}, where blocks are assigned random indices.
The right picture would look like Figure \ref{f:2blocks} under different vertex ordering.
With $\alpha = 0.48$ and $\gamma = 0.048$,
the two cases are indistinguishable to the naked eye.
The structural differences are revealed in Figure \ref{f:2graphpartitions}.
}
} \label{f:2incmats}
\begin{center}
\scalebox{0.4}{\includegraphics{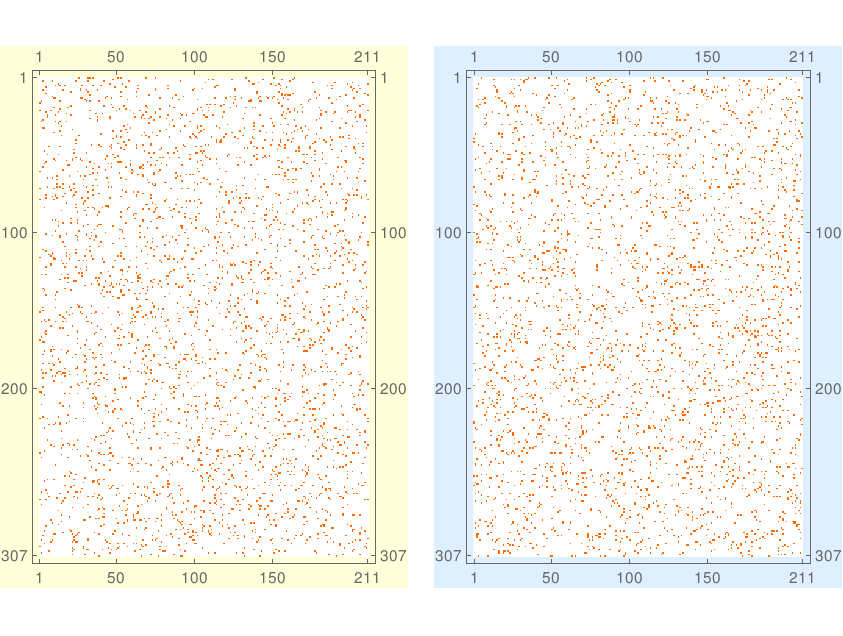}} 
\end{center}
\end{figure}

\subsection{Bernoulli matrix model lacking block structure}\label{s:bmm-noblocks}
Here is an elaborate pseudo-random construction based on modular arithmetic.
Select positive integers $a < b$ and $q:=(a + b)^2$, such that $q, n, m$ are coprime.
The real number
\(
\alpha:=a / \sqrt{q} < 1/2
\)
has the properties
\[
\alpha^2 = \frac{a^2}{q}; \quad (1 - \alpha)^2 = \frac{b^2}{q}.
\]
Partition the residues modulo $q$ into $R_0 \cup R_1 \cup R_2$ in any way so that
\[
|R_0| = q - a^2 - b^2; \quad |R_1| = a^2; \quad |R_2| = b^2.
\]
The Bernoulli parameters $(\gamma_{i,j})$ of $(Z_{i,j})$ are defined as follows.
Fix a reference constant $\gamma \in (0, \alpha)$. Let $f(i,j)$ be
the residue class  of $m (i-1) + n (j-1)$ modulo $q$.
Take
\begin{align*}
\gamma_{i,j}:= & 0, \quad f(i,j) \in R_0 \\
\gamma_{i,j}:= & \frac{\gamma}{\alpha}, \quad f(i,j) \in R_1 \\
\gamma_{i,j}:= & \frac{\gamma}{1-\alpha}, \quad f(i,j) \in R_2. 
\end{align*}
This deterministic scheme ensures that (to within a small discrepancy), 
\begin{enumerate}
\item
There is a pseudo-random set of
$\alpha^2 m n$ cells of the incidence table with parameter $\gamma/ \alpha$. 
\item
There is a disjoint pseudo-random set of
$(1 - \alpha)^2 m n$ cells of the incidence table with parameter $\gamma/ (1 - \alpha)$, 
\item
 The remaining $2 \alpha (1 - \alpha) m n$ of the cells of the incidence table have parameter zero.
\end{enumerate}
Furthermore the proportions of each of these three types of cell are almost the same in every
row and column, thanks to the use of residue classes of $m (i-1) + n (j-1)$ modulo $q$.
Indeed every column total has mean about $n \gamma$ and every row total has mean about $m \gamma$, 
since
 the weighted sum of the parameters in 1, 2, 3 is
\[
 \alpha^2 \cdot \frac{\gamma}{\alpha} + (1 - \alpha)^2 \cdot \frac{\gamma}{1-\alpha}
= \gamma.
\]
The main point is that no causally sparser or denser blocks of the incidence matrix will ever appear,
no matter how the rows and columns are ordered, because the placements of the zero parameters
are essentially different in every row and column. A realization appears on
the left in Figure \ref{f:2incmats}.

\subsection{Bernoulli matrix model with hidden block structure} \label{s:bmm-blocks}
We shall now modify the last example to produce a Bernoulli matrix model $(Z'_{i,j})$,
with block structure, whose Bernoulli parameters $(\gamma'_{i,j})$
are chosen so that the number of index pairs $(i,j)$
for which $\gamma'_{i,j} = \gamma/ \alpha$ is about $\alpha^2 m n$, the number for which
$\gamma'_{i,j} = \gamma/ (1 - \alpha)$ is about $(1 - \alpha)^2 m n$,  and the rest are zero, just as
for the previous case. Recall $\gamma < \alpha < 1/2$.

Let $A$ denote a random sample of $\lfloor \alpha n \rceil$ rows, and let
 $B$ denote a random sample of $\lfloor \alpha m \rceil$ columns.
The random choices of $A$ and $B$ effectively screen the block structure from
visual detection, when $\gamma$ is sufficiently small.
The Bernoulli parameters $(\gamma'_{i,j})$ of $(Z'_{i,j})$ are defined as follows.
\begin{align*}
\gamma'_{i,j}:= & 0, \quad (i,j) \in (A \times B^c) \cup (A^c \times B) \\
\gamma'_{i,j}:= & \frac{\gamma}{\alpha},  \quad (i,j) \in A \times B \\
\gamma'_{i,j}:= & \frac{\gamma}{1-\alpha}, \quad (i,j) \in A^c \times B^c.
\end{align*}

This resembles the example of Section \ref{s:2blocks}, in that a proportion $\alpha^2$
out of the expected total of $\gamma m n$ incidences appear in the $A \times B$ block, and
a proportion $(1 - \alpha)^2$ in the $A^c \times B^c$ block. See Figure \ref{f:2blocks}.
Here too every column total has mean $n \gamma$ and every row total has mean $m \gamma$, 
although the variances are slightly different to those of Section \ref{s:bmm-noblocks}.
In simulations of the models \ref{s:bmm-noblocks} and \ref{s:bmm-blocks},
the resulting incidence matrices are statistically indistinguishable to the 
naked eye for $\gamma / \alpha \leq 0.1$; see Figure \ref{f:2incmats}.

\begin{figure}
\caption{\textit{
The horizontal axis shows the 24 elements of the permutation group $\S_4$,
and the vertical axis shows the frequencies when $\lfloor 3065/4 \rfloor$
four-edge samples were drawn without replacement from the graph corresponding to
the right panel of Figure \ref{f:2incmats}, and converted using
Lehmer codes to elements of $\S_4$. Frequencies fit well to results of multinomial
trials with 24 equally likely outcomes, 
although hidden block structure is present, because the ordering of right and left
vertices conceals the structure of the model of Section \ref{s:bmm-blocks}.
The model of Section \ref{s:bmm-noblocks} gives a similarly uniform distribution.
}
} \label{f:permfrequconceal}
\begin{center}
\scalebox{0.4}{\includegraphics{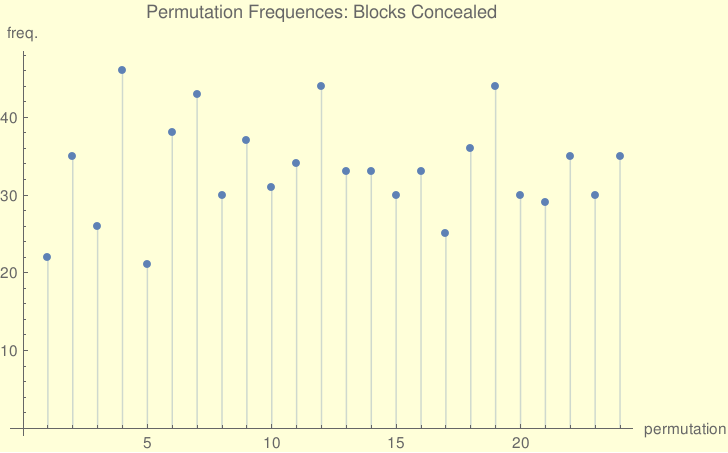} }
\end{center}
\end{figure}

\subsection{Four point test applied to concrete instances} \label{s:4ptstats}
The pseudo-random model of Section \ref{s:bmm-noblocks}, and the hidden block model
of Section \ref{s:bmm-blocks} were instantiated with $n = 307$, $m = 211$, 
$\alpha = 0.48$ and $\gamma = 0.048$. and presented in the left and right panels of
Figure \ref{f:2incmats}, with 3116 and 3065 incidences, respectively.
Block structure is imperceptible on the right
panel, because the index sets $A \subset L$ and $B \subset R$ were selected randomly.

The four point test was applied three times to each matrix. The pseudo-random matrix
scores\footnote{
Random sampling of 4-tuples of edges causes significant random variation
in scores.
}
$14.3, 24.2, 22.1$ were well within the $95$-th percentile $35.1$ of the $\chi^2(23)$
distribution.  Similar scores were observed for the model with
hidden block structure; the frequencies of different permutations are shown in Figure
\ref{f:permfrequconceal}.

Finally the vertex ordering was changed for the model with hidden block structure,
to make vertices in $A$ precede those in $A^c$, and vertices in $B$ precede those in $B^c$.
Such a re-ordering could be inferred from a graph partition algorithm, 
such as the one\footnote{
\texttt{FindGraphPartition} \cite{wol}} that
produced the right pair of globs in Figure \ref{f:2graphpartitions}. Afterwards
three applications of the four point test produced scores $1124, 1138, 1144$, far in the tail
of the $\chi^2(23)$ distribution, and Figure \ref{f:permfrequreveal} shows
the highly imbalanced permutation frequencies.

\subsection{Practical conclusions from the case study}
The case study emphasises that, when block structure is present, the four point test
will reveal it only when the vertices are ordered in a way to tilt the frequencies
of the permutations in $\S_4$. Consecutive applications of the test to the same matrix
will produce answers with a statistical variability which reflects the random
sampling of 4-subsets of the edges.

\begin{figure}
\caption{\textit{
The setting is the same as that of Figure \ref{f:permfrequconceal}, except that
vertex orderings were changed so
vertices in $A$ precede those in $A^c$, and vertices in $B$ precede those in $B^c$,
in the model of Section \ref{s:bmm-blocks}. 
The extreme non-uniformity of the frequencies of the
 24 elements of the permutation group $\S_4$ is apparent, as predicted
by Figure \ref{f:relfreq}.
}
} \label{f:permfrequreveal}
\begin{center}
\scalebox{0.4}{\includegraphics{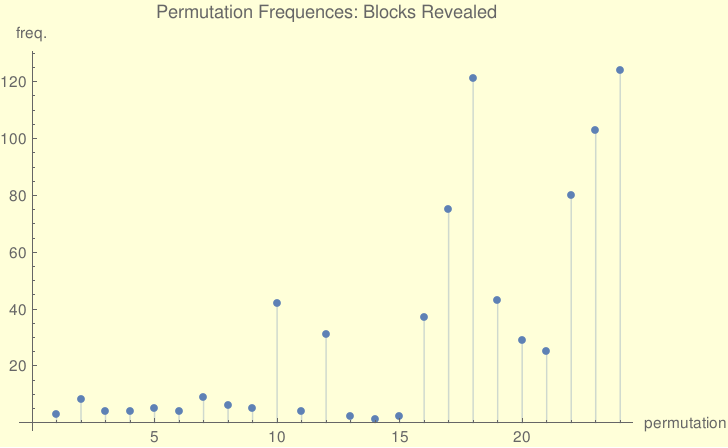} }
\end{center}
\end{figure}


\section{Natural vertex order computation in a bipartite graph} \label{s:naturalorder}

\subsection{Choosing right and left vertex orders}
We have seen in Section \ref{s:orderimportance} that the ordering of left and
right vertices strongly affects the output of the four point test. In this
section we describe one computationally efficient method to select a \textbf{natural order}
for the left vertices and for the right vertices, which tends to highlight
block structure and to boost the Four Point Statistic. See Figure \ref{f:naturalorder}
for a preview. This is not the only possible method: see Section \ref{s:mindeg}.

\subsection{Symmetric linear operators}
Extending the notation of Section \ref{s:notation}, introduce diagonal matrices
\[
\Omega:=\mbox{Diag}(w_1, \ldots, w_n); \quad \Theta:=\mbox{Diag}(d_1, \ldots, d_m).
\]
Rescale the incidence matrix $Z$ to give the $n \times m$ matrix:
\[
M:=\Omega^{-1/2} Z \Theta^{-1/2}.
\]
Introduce two rescaled symmetrized Laplacian operators:
\begin{equation} \label{e:laplacians}
\Delta_L:=I_n - M M^T; \quad \Delta_R:=I_m - M^T M,
\end{equation}
where $I_p$ denotes the $p \times p$ identity matrix. Define vectors
\[
\omega:= \Omega^{-1/2} \mathbf{w} \in \R^n; \quad \theta:= \Theta^{-1/2} \mathbf{d} \in \R^m.
\]
The following well known facts are easily checked by matrix multiplications.
\begin{lemma}
\begin{enumerate}
\item
Both $\omega$ and $\theta$ have norm $\sqrt{N}$ by (\ref{e:degreesums}), and are related by:
\begin{equation} \label{e:reciprocity}
\omega^T M = \theta^T; \quad M \theta = \omega.
\end{equation}
\item
$\omega$ is a positive unnormalized eigenvector of $\Delta_L$ with eigenvalue 0, i.e.
\(
\omega^T \Delta_L = 0;
\)
\item
$\theta$ is a positive unnormalized eigenvector of $\Delta_R$ with eigenvalue 0, i.e.
\(
\theta^T \Delta_R = 0;
\)
\end{enumerate}

\end{lemma}

Assume from now on that $Z$ describes, as in (\ref{e:incidences}), the edges
of a \textit{connected} bipartite graph $G$. It is well known \cite{chu} that connectedness implies that
$\omega$ and $\theta$ are the only eigenvectors with
eigenvalue 0 of $\Delta_L$ and $\Delta_R$, respectively,
and all other eigenvalues are strictly positive.

\subsection{Positive symmetric linear operators}
We shall now shift attention away from Laplacians, towards the positive symmetric operators
$M M^T$ and $M^T M$. We already know that $\omega / \sqrt{N}$ is the unique eigenvector
of eigenvalue 1 for $M M^T$, and likewise $\theta / \sqrt{N}$ for $M^T M$. 
Introduce a new symmetric linear operator on $\R^n$ which composes left multiplication
by $M M^T$ with projection orthogonal to $\omega$, namely
\[
\Gamma_L: \mathbf{y} \mapsto 
M M^T \mathbf{y} - \frac{ \omega^T M M^T \mathbf{y}} {\omega^T \omega}  \omega.
\]
Since $\omega^T \omega = N$ and $\omega^T M M^T = \omega^T$, we may write this operator as
a rank one perturbation of $M M^T$:
\[
\Gamma_L = M M^T - \frac{\omega \omega^T}{N}.
\]
The corresponding operator on $\R^m$ is
\[
\Gamma_R = M^T M - \frac{\theta \theta^T}{N}.
\]
Here are some facts about them, without proof.

\begin{lemma}\label{l:positiveoperators}
Let $r$ denote the rank of $Z$, i.e. of $M$, and suppose the associated bipartite graph $G$
is connected. Each of the operators $\Gamma_L$ and $\Gamma_R$ has the same set of
positive eigenvalues $\lambda_1 \geq \cdots \geq \lambda_{r-1}$, which belong to the set
$(0,1)$. All other eigenvalues of $M M^T$ and $M^T M$ are zero. 
Moreover if $\zeta \in \R^n$ denotes the eigenvector of $\Gamma_L$ associated with
$\lambda_1$, then $\xi:=M^T \zeta \in \R^m$ is an unnormalized eigenvector of $\Gamma_R$ 
associated with $\lambda_1$, and
\[
\omega^T \zeta = 0 = \theta^T \xi.
\]
\end{lemma}

\textbf{Remark: } $\zeta$ and $\xi$ are known as \textbf{Fiedler Vectors} for
the induced graphs on the left and right vertex sets, respectively.

Without loss of generality, suppose $n \geq m$; otherwise work instead with $Z^T$.
Hence we put emphasis on $\Gamma_L$, and derive results for $\Gamma_R$ from 
Lemma \ref{l:positiveoperators}.

\subsection{Power method}

\begin{proposition}[POWER METHOD] \label{p:powermethod}
Take a random vector $\eta \in \R^n$ whose components are independent normal$(0,1)$
random variables. Project $\eta$ orthogonal to $\omega$,
and rescale to norm 1 to obtain $\mathbf{y}_{(0)}$.
Iterate for $t \geq 1$:
\begin{equation} \label{e:iterate}
\mathbf{z}_{(t)} = M (M^T \mathbf{y}_{(t-1)} ) - 
\frac{\omega \cdot \mathbf{y}_{(t-1)} } {N} \omega; \quad
\mathbf{y}_{(t)} = \frac{\mathbf{z}_{(t)} }{ \|\mathbf{z}_{(t)}\| }.
\end{equation}
Let $\phi_t \in [0, \pi/2]$ denote the angle such that $\cos{\phi_t} = |\zeta \cdot \mathbf{y}_{(t)}|$.
The event $\cos{\phi_0} \neq 0$ has probability 1, and in that case 
\begin{equation} \label{e:convrate}
|\sin{\phi_t}| \leq  (\lambda_2 / \lambda_1)^t \tan{\phi_0}.
\end{equation}
This implies that, with probability 1, 
\(
 \lim_{t \to \infty} \mathbf{y}_{(t)}
\)
exists and is equal to $\zeta$ or to $-\zeta$. Provided $\lambda_2 < \lambda_1$, the convergence
occurs at an exponential rate.
\end{proposition}

\begin{proof}
This iterative scheme is the power method decribed in
Golub \& Van Loan \cite[Theorem 8.2.1]{gol} for the computation of
the eigenvector $\zeta$ with top eigenvalue $\lambda_1$ of the symmetric
linear operator $\Gamma_L$. The cited theorem proves the bound on $|\sin{\phi_t}|$.
\end{proof}

\textbf{Implementation issues: } 
\begin{enumerate}
\item
Since an approximation suffices, we propose to fix some $\delta > 0$ and to
stop the iteration (\ref{e:iterate}) at the first $t$ for which
\[
\|\mathbf{y}_{(t)} - \mathbf{y}_{(t-1)} \| < \delta.
\]
For a given spectrum, (\ref{e:convrate}) implies that $O(\log{(1/\delta)})$
matrix multiplies will suffice, each of which is $O(N)$ work. We observe in practice that
if the local structure of $G$ remains statistically similar as $N$ increases, the
number of iterations before stopping does not vary as $N$ increases, implying that
total work is $O(N \log{(1/\delta)})$. A crude upper bound $\hat{d}$ for graph diameter can be
obtained by selecting a left vertex uniformly at random, and taking $\hat{d}$ to be twice
the number of steps of breadth first search needed to cover the graph entirely.
In the absence of an estimate for $\lambda_2 / \lambda_1$, we observed that in sparse graphs
$2 \hat{d}$ iterations were sufficient for convergence when $\delta \approx 0.05$.
For more on the relation between graph spectrum and graph diameter,
see Chung \cite[Ch. 3]{chu}. The heuristic claim is that $O(N \hat{d})$ work
suffices for computing an adequate natural order.

\item
Probabilistic arguments show that the random variable $\tan{\phi_0}$ in the upper bound 
(\ref{e:convrate}) is $O(\sqrt{n})$.

\item
In experiments, 
the ratio $(1/t)\log{|\sin{\phi_t}|}$ is typically less than $\lambda_2 / \lambda_1$,
making the convergence faster than that implied by (\ref{e:convrate}).

\item
We have phrased the iteration (\ref{e:iterate}) in terms of the symmetric operator
$\Gamma_L$ in order to appeal to the literature on the symmetric eigenvalue problem.
In computational implementation the matrix $Z$ is typically given by two jagged arrays,
one giving a look-up by row, and the other giving a look-up by column. 
The iteration  (\ref{e:iterate}) can be implemented under the rescaling $x_i:=y_i/\sqrt{w_i}$:
\[
\mathbf{x} \mapsto \Omega ^{-1} Z (\Theta^{-1} Z^T \mathbf{x})
- \mathbf{1} \frac{\mathbf{w \cdot x}}{N},
\]
where $\mathbf{1}$ is the all ones vector.
The normalization step need not be performed in the $\| \cdot\|_2$
norm. It can, for example, be performed in the $\| \cdot\|_1$ norm instead.

\end{enumerate}

\subsection{Definition of natural  order}
\begin{definition}\label{d:naturalorder}
The left vertex set $L$ is in \textbf{natural order} if vertices are 
in decreasing or increasing order of the corresponding components of the eigenvector
$\zeta$, described in Lemma \ref{l:positiveoperators}. Likewise components of $\xi$
supply a natural order for the right vertex set $R$.
\end{definition}

In this definition we do not insist that $\zeta$ or $\xi$ be computed precisely.
Indeed an approximation, constructed as in Proposition \ref{p:powermethod}, suffices.

See Figure \ref{f:naturalorder} for an illustration of an incidence matrix transformed into
natural order of left and right vertices.

\subsection{Scaling behavior in natural order and four point test computations} \label{s;scaling}
\begin{figure}
\caption{\textit{
$4000 \times 3199$ incidence matrix shown at left, constructed as in Section \ref{s;scaling},
 is presented in natural order on the right.
Typical four point total variation scores (\ref{e:imbalance})
are 0.26 on the left, and 0.43 on the right. When similar matrix generation schemes are
applied at different scales, keeping the matrix aspect ratio fixed,
these scores (\ref{e:imbalance}) seem to be scale-free.
}
} \label{f:naturalorder}
\begin{center}
\scalebox{0.5}{\includegraphics{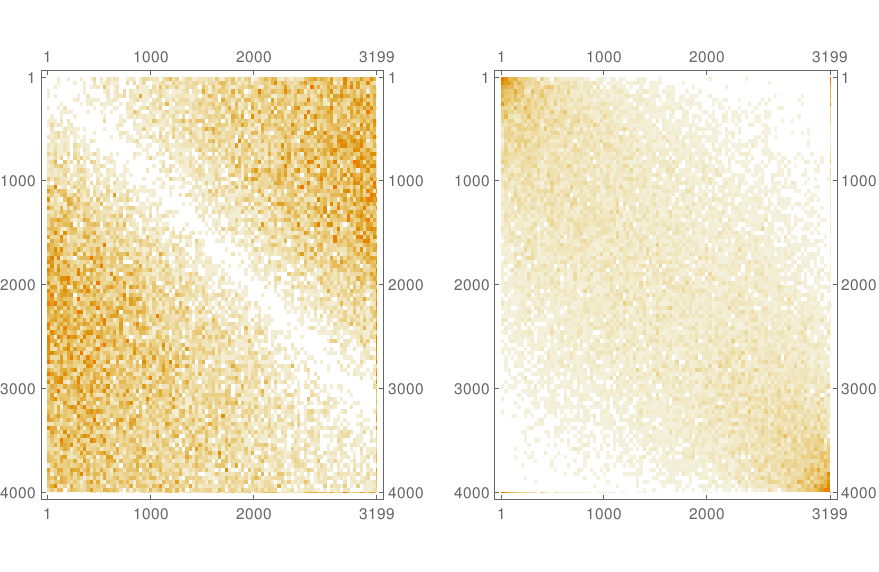}
}
\end{center}
\end{figure} 

The natural order and four point test computations have been implemented 
both in a \textit{Mathematica} prototype and in a performant Java 10
package called \texttt{QuantifyBipartiteBlockStructure}.

We simulated some $k$-regular random hypergraphs on $n$ vertices, where
the $k$ vertices in hyperedge $i$ were not picked uniformly, but were a weighted
sample using weight $1 + |i - j|$ for vertex $j$, which tends 
to force incidences away from the diagonal. For $s \in \{50, 100, 200, 400\}$,
we simulated two instances of such random hypergraphs for parameter choices
$k = 7$, $n = 10 s$, $m = 8 s$. Empty columns were discarded. Figure \ref{f:naturalorder}
shows one of the largest matrices, both before and after the natural order computation.

The four point total variation score (\ref{e:imbalance})  was always in the $0.25 < D_4 < 0.27$
range for the raw matrix, and in the $0.42 < D_4 < 0.44$ range for the naturally ordered matrix,
regardless of scale. $D_4$ varied as much between two matrices of the same size as it did between
two matrices of different sizes.
This suggests the possibility of proving limit results for values of $D_4$
as $n, m \to \infty$ under suitable assumptions about the matrix generation mechanism.

The four point chi-squared statistic
(\ref{e:4ptteststat}) appears to scale in proportion to $s$ in these examples. Karl Rohe \cite{roh} wrote an implementation which reveals
 that the significance level
of the $\chi^2$ statistic is incorrect if all
edges of the graph are used both (1) to derive
an approximate Fiedler vector, and (2) to perform the
Four Point Test. This weakness can be avoided by
partitioning
the edge set randomly into two subsets,
one of which is used for the former, and the other
for the latter.

\begin{table}
\footnotesize
\centering
\begin{tabular}{ | l | l | l |l | l | l | l | l | l | }
\hline
Review Set &  \# edges & \# left & \# right & TV & giant & TV-NO & 4PT &  N.O. \\ \hline
Dig. Music &  0.836M &  0.478M &  0.266M &  \textbf{0.596} &  0.703M &  0.507 &  0.484s &  7.86s \\
Android&  2.64M &  1.32M &  61.3K &  \textbf{0.487} &  2.63M &  0.350 &  1.47s &  18.9s \\
Movies/TV &  4.607M &  2.089M &  0.201M &  \textbf{0.418} &  4.573M &  0.321 &  2.86s &  57.8s \\
Electronics &  7.824M &  4.20M &  0.476M &  \textbf{0.492} &  7.73M &  0.397 &  3.44s &  76.8s \\
Books &  22.5M &  8.03M &  2.33M &  0.308 &  22.3M &  \textbf{0.349} &  14.6s &  226s \\
\hline
\end{tabular}
\caption{\textit{Four point test applied to five Amazon Reviews data sets, both
before and after natural ordering of left and right vertex sets.} The number of edges,
and the number in the giant component (in millions) are shown in columns 1 and 5, respectively.
Here TV and TV-NO refer to the total variation statistic (\ref{e:imbalance}) computed in the original ordering, and in the natural ordering (on the giant component), respectively.
The last two columns show execution times of single-threaded Java 10
code for the four point test, and for the natural ordering calculation, respectively.
These times scale linearly in the number of edges.
}
\label{t:amazon}
\end{table}

\subsection{Large natural order and four point test computations}

We performed four point test and natural ordering computations on five sets of
Amazon Reviews data\footnote{
jmcauley.ucsd.edu/data/amazon}, as shown in Table \ref{t:amazon}.
In all cases left vertices were reviewers, and right vertices were products
of a specific type. Reading the data took longer than performing the four
point test, whose time scaled linearly in the number of edges, as expected;
see Figure \ref{f:timings}. It is noteworthy that execution times for ratural ordering,
which typically required about 25 iterations of the power method, also scaled
linearly in the number of edges.

Only for Amazon reviews of books did the natural ordering improve the score in the four point 
total variation statistic (\ref{e:imbalance}). For the other four product categories,
 the original order yields a higher score. The high scores suggest that, for example, music
tracks fall into music genres, and reviewers of one genre do not tend to review other genres.
This effect is least for books: some reviewers may rate multiple types of literature.

\begin{figure}
\caption{\textit{Log-log plots show
execution times of the four point test and natural order computations
scale linearly on the Amazon Reviews data sets, whose sizes are shown in 
Table \ref{t:amazon}.}
} \label{f:timings}
\begin{center}
\scalebox{0.32}{\includegraphics{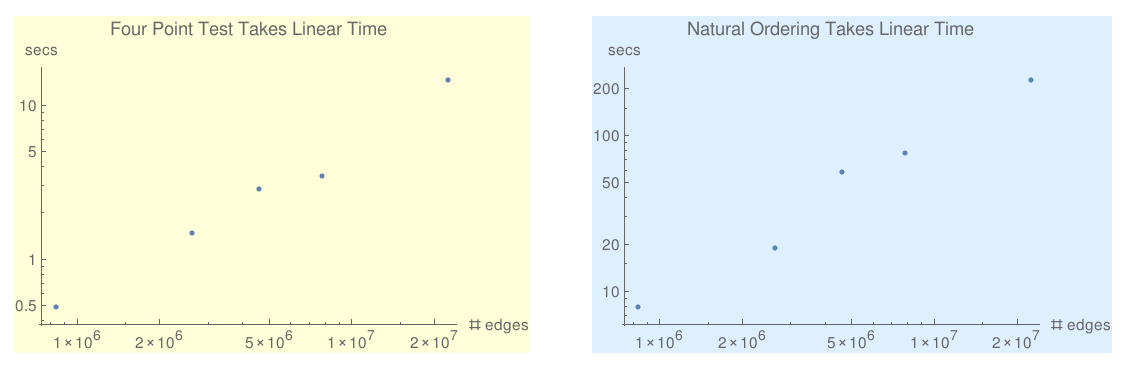}
}
\end{center}
\end{figure}

\section{Correspondence between permutations and bipartite graphs} \label{s:perm2graph}

The methods of this section lead to a proof of Theorem \ref{t:abf4}.

\subsection{Random permutation generates bipartite graph: fixed vertex degrees}\label{s:pgfix}

This section is inspired by the half-edge construction due to Wormald, and the configuration model in 
Bollob{\'a}s \cite[Section II.4]{bol}.
A totally ordered left vertex set $L:=\{u_1, u_2, \ldots, u_n\}$ and a 
totally ordered right vertex set 
$R:=\{v_1, v_2, \ldots, v_r\}$ are given. 
Fix a left vertex degree vector $\mathbf{w}:=(w_1, \ldots, w_n)$, and right vertex degree vector
 $\mathbf{d}:=(d_1, \ldots, d_m)$ in advance, where both vectors sum to $N$.
It is required that $u_i$ has degree $w_i \geq 1$,
and $v_j$ has degree $d_j \geq 1$. It is convenient to introduce the partial sums
\[
W_i:=\sum_{i'=1}^i w_{i'}; \quad D_j:= \sum_{j'=1}^j d_{j'},
\]
with $W_0 = D_0 = 0$, $W_n = D_m = N$.

Construct two sequences $h^L:=(h_{p}^L)$ and $h^R:=(h_{q}^R)$ of vertex labels,
both of length $N$, where $h_{p}^L = u_i$ when $W_{i-1} < p \leq W_i$, and 
$h_{q}^R = v_j$ when $D_{j-1} < q \leq D_j$.
Thus $h^L$ contains $w_1$ symbols referring to $u_1$, then $w_2$ symbols referring to  $u_2$,
and so on:
\[
h^L:=(\overbrace{u_1, \ldots, u_1}, \overbrace{u_2, \ldots, u_2}, \ldots,
\overbrace{u_n, \ldots, u_n}),
\]
while $h^R$ contains $d_1$ symbols referring to  $v_1$, 
then $d_2$ symbols referring to $v_2$,
and so on. 
We call $h^L$ and $h^R$ left and right \textbf{half-edge vectors},
respectively.

\begin{definition}\label{d:perm2graph}
Given left and right half-edge vectors $h^L$ and $h^R$, respectively, of length $N$,
the bipartite multigraph $B_{\pi}:=(L \cup R, E)$ induced by a permutation $\pi \in S_N$
is the graph whose edge set $E$ consists of the pairs
\begin{equation} \label{e:halfedgecon}
E:=\{(h_{p}^L, h_{\pi(p)}^R), 1 \leq p \leq N\}.
\end{equation}
\end{definition}

We estimate in Section \ref{s:duplicates} the expected number of duplicate edges in $E$.
Blanchet \& Stauffer \cite{bla} give necessary and sufficient conditions,
also proved in Janson \cite{jan2}, for
the asymptotic probability of obtaining a simple graph to be positive.

The following lemma is nearly a tautology, given the construction (\ref{e:halfedgecon}).

\begin{lemma}\label{l:abffromperm}
Suppose for each $k \geq 1$, $h^{L(k)}$ and $h^{R(k)}$ are left and right half-edge vectors
of the same length $N(k)$, where $L(k)$ and $R(k)$ are the sets of distinct labels
occurring in the respective vectors. Take $B_{\pi}(k)$ to be bipartite (multi)graph
on $L(k) \cup R(k)$ induced, as in (\ref{e:halfedgecon}), by a 
uniform random permutation $\pi_k \in S_{N(k)}$. 
If $|L(k)| \to \infty$, $|R(k)| \to \infty$, and $N(k) \to \infty$. 
then $(B_{\pi}(k))$ is asymptotically block-free.
\end{lemma}

\begin{proof}
Fix $s \geq 2$. For any $k$ such that $N(k) \geq s$,
select $s$ edges uniformly at random, say
\[
\{(h_{p_1}^L, h_{\pi(p_1)}^R), \ldots, (h_{p_s}^L, h_{\pi(p_s)}^R) \},
\]
where for brevity we have dropped the index $k$ from the notation.
The left endpoints $(h_{p_1}^L, \ldots, h_{p_s}^L)$ are already in increasing order.
Since the permutation $\pi$ is uniformly random, the $s$ right endpoints
$(h_{\pi(p_1)}^R, \ldots, h_{\pi(p_s)}^R)$ are ordered 
uniformly at random. Thus the every $s$, the sequence $(B_{\pi}(k))$ is 
asymptotically block-free of order $s$.
\end{proof}

\subsection{Inversion of the half edge construction} \label{s:inversion}
We shall now describe a way to invert Definition \ref{d:perm2graph}, so as to be able to produce a
permutation of $N$ symbols from a bipartite graph with $N$ edges. This will be 
used in the proof of Theorem \ref{t:abf4}.

Let us elaborate on the construction of total orders on edges,
introduced in Section \ref{s:4pttest}.
Fix an arbitrary \textbf{total order} $\prec_L$ on the edges, $e_1 \prec_L \cdots \prec_L e_N$,
with the property that, for all $1 \leq i < i' \leq n$,
\begin{equation} \label{e:leftorder}
(u_i, v) \prec_L (u_{i'}, v'), \quad \forall v, v' \in R.
\end{equation}
In other words, the order on edges is consistent with the order on left vertices.
Next generate $N$ i.i.d. Uniform$(0,1)$ random variables $U_1', U_2', \ldots, U_N'$,
which will be used as tie breakers in the following way. Extend
the right half-edges $h^R$ above, i.e.
\[
h^R:=(\overbrace{v_1, \ldots, v_1}, \overbrace{v_2, \ldots, v_2}, \ldots,
\overbrace{v_m, \ldots, v_m})
\]
to a series of $N$ pairs
\begin{equation} \label{e:auxrandom}
(\overbrace{(v_1,U_1'), \ldots, (v_1, U_{D_1}')}, 
\overbrace{(v_2, U_{D_1 + 1}'), \ldots, (v_2, U_{D_2}')}, \ldots,
\overbrace{(v_m, U_{D_{m-1} + 1}'), \ldots, (v_m, U_{D_m}')}).
\end{equation}
This yields another \textbf{total order} $\prec_R$ on the edges, namely lexicographic ordering
 using first the ordering on the $(v_j)$, then the ordering on the $(U'_k)$.
In other words, 
\begin{equation} \label{e:rightorder}
e_p:=(u, v_j) \prec_R e_{q}:=(u', v_{j'})
\end{equation}
if either $j < j'$, or else $j = j'$ and $U'_p < U'_{q}$.

\begin{definition} \label{d:graph2perm}
Suppose $G:=(L \cup R, E)$ is a bipartite (multi)graph with $N:=|E|$ edges.
The left total order $(E, \prec_L)$ (\ref{e:leftorder}) and
the random right total order $(E, \prec_R)$ (\ref{e:rightorder})  
combine to induce a random permutation $\pi_G \in S_N$ by
\[
e_k = e'_{\pi_G(k)}
\]
where $e_1 \prec_L \cdots \prec_L e_N$ are the left-ordered edges, and
$e'_1 \prec_R \cdots \prec_R e_N'$ are the right-ordered edges
(subject to the randomization (\ref{e:auxrandom}) to break ties).
\end{definition}

From the constructions above, the following Inversion Lemma is a tautology.

\begin{lemma}[INVERSION] \label{l:inversion}
If the random permutation $\pi_G \in S_N$ of Definition \ref{d:graph2perm}
is applied in the half-edge construction of Definition \ref{d:perm2graph}
to the half-edge sequences $h^L$ and $h^R$, then the resulting edge set 
(\ref{e:halfedgecon}) coincides with 
the original edge set $E$ of the graph $G$.
\end{lemma}

\subsection{Permutation terminology: Property $P(k)$} \label{s:P(k)}

This terminology is reproduced from \cite{kra}.
Let $S_k$ consist of permutations on $[k] :=
\{1, . . . , k\}$. 
We view each $\pi \in S_k$ as a bijection $\pi: [k] \to [k]$, and we say that the length of $\pi$ is $k$.
 For $\pi \in S_k$ and $\tau \in S_m$ with $k \leq m$, let $t(\pi, \tau)$
 be the probability that a random $k$-point subset $X$ of $[m]$ 
induces a permutation isomorphic to $\pi$ (that is,
$\tau(x_i) \leq \tau(x_j)$ iff $\pi(i) \leq \pi(j)$ where $X$ consists of 
$x_1 < \ldots < x_k$). A sequence $\{\tau_j \}$
of permutations is said to have Property $P(k)$ if their lengths tend to $\infty$ and 
$t(\pi, \tau_j) = 1/k! + o(1)$ for every
$\pi \in S_k$ . It is easy to see that $P(k + 1)$ implies $P(k)$.

\subsection{Proof of Theorem \ref{t:abf4}}

Take a sequence
$(G_k):=((L_k \cup R_k, E_k))$ of random bipartite (multi)graphs,
where $|L_k| \to \infty$, $|R_k| \to \infty$, 
and $|E_k| \to \infty$,
which is asymptotically block-free of order 4.

Apply Definition \ref{d:graph2perm} to convert each graph $G_k$ into
a random permutation $\pi_{G_k}$ of $|E_k|$ symbols.
From asymptotically block-freeness of order 4, and the auxiliary
randomization (\ref{e:auxrandom}), it follows that Property $P(4)$
holds for the sequence $(\pi_{G_k})$, in the sense of Section \ref{s:P(k)}.
Theorem 1 of Kr{\'a}l$'$ \& Pikhurko \cite{kra} shows that Property $P(s)$ holds
for all $s$. Together with Lemma \ref{l:inversion}, this implies $(G_k)$ is asymptotically block-free of order $s$,
for all $s \geq 2$, as desired.


\section{Open problems}

In this new area of research, many topics remain to be explored.

\subsection{Directed non-bipartite graphs }
The four point test computation of Section \ref{s:4pttest} makes sense not only for bipartite graphs,
but for any directed graph on an ordered vertex set. What exactly is the scope and
meaning of the test statistics (\ref{e:4ptteststat}), (\ref{e:imbalance}) in the directed non-bipartite case?

\subsection{Vertex exchangeability and edge exchangeability}

 Caron \& Fox \cite{car} present constructions of random bipartite graphs
where the left vertices are exchangeable, and the right vertices are exchangeable.
A general case is described by
Borgs, Chayes, Cohn and Holden \cite{cha}.
Cai, Campbell \& Broderick \cite{cai} and
Crane \& Dempsey \cite{cra} have defined the notion of
an edge-exchangeable graph sequence.
What happens when one applies the four point test to
vertex-exchangeable or edge-exchangeable graph sequences?

\subsection{Minimum degree instead of natural order}\label{s:mindeg}
The natural order defined in Section \ref{s:naturalorder} is neither the only,  
nor the cheapest, approach to ordering rows and columns of a sparse matrix in
order to expose something resembling block structure. For example, Duff et al \cite{duf}
describe the minimum degree algorithm. This starts with all rows declared active, and terminates
when no active rows remain. Active degree of column $v$ means the number of incidences of column
$v$ with active rows.
Iterate as follows:
\begin{enumerate}
\item
Select some column $v$ uniformly at random from those of minimum non-zero 
active degree, and place it next in the column ordering.
\item
Active rows incident to $v$ are placed next in the row ordering, and are then declared inactive.
\item
Update active degrees of columns by subtracting counts of incidences with newly inactive rows.
\end{enumerate}
Column labels left over when active rows are exhausted are placed in arbitrary order, after the others.
We would like to know whether applying minimum degree to some kinds of sparse matrices leads to
higher or lower 4PT-TV scores than applying natural order.

\subsection{Discrepancy measures in bipartite graphs}

Given vertex sets $U \subset L$ and $V \subset R$
in a directed bipartite graph $G:=(L \cup R, E)$,
let $Z_{U,V}$ count the set $E(U,V) \subset E$ of edges between $U$ and $V$:
\[
Z_{U,V}:= \sum_{i: u_i \in U} \sum_{j: v_j \in V} Z_{i,j} = |E(U,V)|.
\]
The total degree of vertices in $U$, and in $V$, respectively, is
\[
W(U):=Z_{U,R}; \quad D(V):=Z_{L,V}.
\]
Motivated by the notion of discrepancy, which gives one of the equivalent definitions
 of a quasirandom permutation \cite{kra},
define the \textbf{discrepancy} in the bipartite graph $G$ to be the random variable
\[
\Delta(G);=\max_{U \subset L, V \subset R} \left|
Z_{U,V} - \frac{W(U) D(V)}  { N }
\right|.
\]
The open problem is to give computable bounds on the discrepancy of a sequence
of random bipartite graphs which are asymptotically block-free in the sense of Definition
\ref{d:abfos}.
Possibly such bounds may be derived from concentration inequalities such as 
are found in Janson \cite[Theorem 8]{jan}.

\subsection{Relation to quasirandom hypergraphs}
Quasirandom hypergraphs are those which have the properties one would expect to find in ``truly'' random hypergraphs, in which a $k$-edge contains $k$ vertices selected uniformly without replacement,
and all $k$-edges are statistically independent.
Shapira \& Yuster \cite{sha}, Lenz and Mubayi \cite{len}, \cite{len2}, 
and other authors cited therein, study quasirandomness
in sequences $(H_n)$ of dense $k$-uniform hypergraphs, meaning that, for some $p \in (0,1)$,
the number $|E(H_n[U])|$ of hyperedges with vertices inside any $U\subset V(H_n)$ is
\[
p \binom{|U|}{k} + o(n^k).
\]
The study of quasirandom structures lies at the core of recent proofs of Szemer{\' e}di's Theorem (see \cite{bol0}) 
obtained by Gowers, and by R{\" o}dl et al.
We would like to clarify how this theory of dense quasirandom hypergraphs
interacts with the approach to sparse hypergraphs (viewed in terms of bipartite graphs and
quasirandom permutations) that we have taken here.

\appendix


\section{Appendix: likelihood ratio statistic for sparse binary contingency tables}

\subsection{Purpose}
This section is intended to assuage the concerns of statisticians
for whom tests of association in binary contingency table necessarily involve
the likelihood ratio statistic.
Koehler \cite{koe} considers the problem of testing for independence of rows and columns
in a sequence of expanding two-dimensional contingency tables,
where the $k$-th table in the sequence has $N_k$ entries distributed among
$r(k)$ rows and $c(k)$ columns. By contrast, traditional contingency table analysis
considers a table of fixed dimensions as sample size increases; see Agresti \cite{agr}.

We will see that the likelihood ratio statistic detects association
between row degree and column degree in a binary contingency table, 
and its variance detects non-uniform incidence rates, but it
does not detect block structure in sparse tables, as the Section
\ref{s:orderimportance} examples will show.

\subsection{Likelihood ratio statistic in the Bernoulli matrix model}

For simplicity consider first the \textbf{Bernoulli matrix model} of a binary contingency table, where
incidence $Z_{i,j}=1_{u_i \sim v_j}$ is Bernoulli$(\gamma_{i,j})$, for some constants $(\gamma_{i,j})$
with values in $[0, 1]$. Let $K$ denote the set of index pairs $(i,j)$ for which $\gamma_{i,j} > 0$.
Define a log odds ratio
\begin{equation} \label{e:logodds}
\lambda_{i,j}:=\log{\frac{1 - \gamma_{i,j}}{\gamma_{i,j}}}, \quad (i,j) \in K.
\end{equation}
and a normalizing constant
\begin{equation} \label{e:normalizellf}
\Omega:= \Omega(\gamma_{i,j}):=  \sqrt{\left(
\sum_{(r,s) \in K} \gamma_{r,s} (1 - \gamma_{r,s}) \lambda_{r,s}^2 
\right)}.
\end{equation}
The null hypothesis $H_0$ states that the $(Z_{i,j})$ are independent.
The standardized version of the likelihood ratio statistic for testing $H_0$ is
\begin{equation} \label{e:edgellf}
\xi:= \sum_{(i, j) \in K} \xi_{i,j}; \quad
\xi_{i,j}:=\frac{\lambda_{i,j} } { \Omega } (Z_{i,j} - \gamma_{i,j}).
\end{equation}
This is an affine function of the vector 
of log likelihoods for the $(Z_{i,j})$, as we see from the 
following elementary Lemma:
\begin{lemma} \label{l:loglik}
Suppose $Z\sim$Bernoulli($\theta$) with $0 \leq \theta < 1/2$. Take
\[
X:=(Z-\theta) \log{\frac{1-\theta} {\theta}}= L(Z) - H(Z),
\]
where 
\(
L(Z) = - Z \log{\theta} - (1-Z) \log{(1-\theta)}
\)
is the negative log likelihood,
and $H(Z)$ is the Shannon entropy of $Z$. Then $\E[X]=0$, and $|X| \leq - log{\frac{1-\theta}{\theta}}$. 
\end{lemma}

For convenience in normal approximation, the
log likelihood $\xi$ is scaled so $\E[\xi] = 0$, $\mbox{Var}[\xi] = 1$.

\begin{proposition}[ASYMPTOTIC NORMALITY]\label{p:clt-bmm}
Consider a sequence of Bernoulli matrix models $(Z^k_{i,j})$ as $k \to \infty$, with
$Z^k_{i,j} \sim$ Bernoulli$(\gamma^k_{i,j})$, and index sets $K^k:=\{(i,j): \gamma^k_{i,j} > 0\}$.
Suppose the
log odds ratio $\lambda_{i,j}^k$ as in  (\ref{e:logodds}), divided by the normalizing constant
(\ref{e:normalizellf}), has the property that, as $k \to \infty$,
\begin{equation}\label{e:maxrate}
 \frac{ \max_{(i,j) \in K^k} { \{\lambda_{i,j}^k \} } }{\Omega(\gamma_{r,s}^k)} \to 0.   
\end{equation}
Then the random variables $( \xi_{i,j}^k)$ in (\ref{e:edgellf}) satisfy
Lindeberg's condition:
\begin{equation} \label{e:lindeberg}
\sum_{(i,j) \in K^k} \E[\left( \xi_{i,j}^k \right)^2 ; |\xi_{i,j}^k | > \epsilon] \to 0.
\quad \forall \epsilon > 0.
\end{equation}
Hence the rescaled likelihood ratio statistic $\xi^k$ in (\ref{e:edgellf}) converges in distribution
to standard normal by Lindeberg's central limit theorem.
\end{proposition}

\textbf{Remark: } Compare the assumption (\ref{e:maxrate}),
where cell frequencies tend to zero, and
indeed may be $O((r(k) + c(k))^{-1})$, with
 Koehler's \cite{koe}, wherein all expected cell frequencies are
bounded below by a strictly positive constant as $k \to \infty$.

\begin{proof}
The right side of (\ref{e:edgellf}) is
a sum of independent random variables, and this sum has mean zero and variance 1. 
The final assertion about the central limit theorem follows from
Kallenberg \cite[Theorem 5.12]{kal}, once we have verified 
\textit{Lindeberg's condition} (\ref{e:lindeberg}).

For brevity, drop the superfix $k$ from the notation, and study a fixed $k$.
Let 
\[
\delta:= \max_{(i,j) \in K} \frac{\lambda_{i,j}}{\Omega}.
\]
Since $Z_{i,j} \sim$ Bernoulli$(\gamma_{i,j})$, and 
$\xi_{i,j}:=\lambda_{i,j} (Z_{i,j} - \gamma_{i,j})/\Omega $,
it follows that
\[
|\xi_{i,j} | \leq |\lambda_{i,j} / \Omega|  \leq \delta.
\]
Suppose $\epsilon > 0$ is fixed. Choose $k$ sufficiently large that
\[
\max_{(i,j) \in K^k} \frac{\lambda_{i,j}^k}{\Omega(\gamma_{r,s}^k)} < \epsilon.
\]
For such $k$, we have $ \delta < \epsilon$, and hence
\[
\E[\left( \xi_{i,j}^k \right)^2 ; |\xi_{i,j}^k | > \epsilon] = 0, \quad \forall (i,j).
\]
Thus (\ref{e:lindeberg}) is established.
\end{proof}

\subsection{Examples to show log likelihood fails to detect blocks} \label{s:lrs2models}
Recall the models of Sections \ref{s:bmm-noblocks} and \ref{s:bmm-blocks}.
They were designed as Bernoulli matrix models in which the Bernoulli
parameters $\gamma/\alpha$, $\gamma/(1 - \alpha)$, and 0 appear with
similar frequencies, but with different structural organization. Let us study the
likelihood ratio statistic for the two models.

There are two cases to consider, depending on whether the matrices $(\gamma_{i,j})$ and
$(\gamma'_{i,j})$ of Bernoulli parameters for the two models are known or unknown.

\textbf{Parameters known: }
 Consider the ingredients (\ref{e:logodds}), (\ref{e:normalizellf}), (\ref{e:edgellf})
from which the likelihood ratio statistic is derived. These ingredients are 
essentially the same in models of Sections \ref{s:bmm-noblocks} and \ref{s:bmm-blocks},
the only difference being in the ordering of labels, which is irrelevant when summing.
Model \ref{s:bmm-blocks} has block structure while model \ref{s:bmm-noblocks} does not. 
Knowledge of the parameter matrix $(\gamma'_{i,j})$ reveals block structure, but the
likelihood ratio statistic itself does not reveal block structure.

\textbf{Parameters unknown: } Given a pair of incidence matrices $(Z_{i,j})$ and $(Z'_{i,j})$,
generated according to models \ref{s:bmm-noblocks} and \ref{s:bmm-blocks} respectively,
the statistician will observe that, for both matrices,
 the column totals look like samples from Binomial$(n, \gamma)$,
while row totals look like samples from Binomial$(m, \gamma)$, 
where the unknown parameter could be estimated as 
$\hat{\gamma}:=\sum_{i,j}Z_{i,j}/(mn)$ or $\sum_{i,j}Z'_{i,j}/(mn)$.
The statistician will then compute the
likelihood ratio statistic $\xi$, as in (\ref{e:edgellf}), based on the 
model $\gamma_{i,j} = \hat{\gamma}$ for all $i, j$. Some cancellation occurs,
and
\[
\xi = \frac{1}{ \sqrt{  m n \hat{\gamma} (1 - \hat{\gamma} ) } } 
\sum_{i,j} (Z_{i,j} - \hat{\gamma}).
\]
This has mean zero, by choice of $\hat{\gamma}$,
and the question comes down to testing for excessive variance.
For example, one could test whether
\[
\sum_{i,j} \frac{(Z_{i,j} - \hat{\gamma})^2} { \hat{\gamma} (1 - \hat{\gamma} ) }
\]
exceeds the $1 -p$ quantile of $\chi^2(m n -1)$, for a size $p$ test.
The test results for models \ref{s:bmm-noblocks} and \ref{s:bmm-blocks}
will be similar; rejection of the hypothesis that all entries are i.i.d. Bernoulli
is likely in both cases. The fact that model \ref{s:bmm-blocks} has block structure,
whereas model \ref{s:bmm-noblocks} does not, is not discovered by this test.

\subsection{Regular bipartite subgraphs}
The special case of a $(w,d)$-regular bipartite graph is the one where every left vertex has degree $w$, every right vertex has degree $d$, and thus $w n = N = d m$.

Consider a sparse case where $n, m$ are large, $N = O(\max{ \{m ,n\} })$, and
the set of distinct values of the $(w_i)$ and $(d_j)$ is
$O(\log{N})$. It makes sense to view bipartite graph $G$ as a collection
 of regular bipartite graphs $\{(L_w \cup R_d, E_{w,d})_{w,d}\}$,
where $L_w$ consists of left vertices of degree $w$,
$R_d$ consists of left vertices of degree $d$, and
$E_{w,d}$ consists of edges $(u_i, v_j)$ with $w_i = w$, $d_j = d$. Thus
\[
L:=\bigcup_w L_w; \quad
R:=\bigcup_d R_d; \quad
E:= \bigcup_{w,d} E_{w,d}
\]
This is equivalent to organizing the 0-1 matrix $Z$ into blocks, according to
row sum and column sum.
Fix a row weight $w$ and column degree $d$. The total number of incidences in
the $(w,d)$ block is $N_{w,d}$ as in (\ref{e:jointdegmat}).
These totals may be expressed in a contingency table of the form:
\begin{center}
\begin{tabular}{l | r |r | r}
 & $d = 2$ & $d = 3$ & $\cdots$ \\ \hline
$w = 2$ & $N_{2,2}$ & $N_{2,3}$  & $\cdots$ \\ \hline
$w = 3$ & $N_{3,2}$ & $N_{3,3}$  & $\cdots$ \\ \hline
$\cdots$  & $\cdots$  & $\cdots$   & $\cdots$ \\ \hline
\end{tabular}
\end{center}

\subsection{Association of left and right vertex degrees}

Suppose the $(\gamma_{i,j})$ are unknown.
For a cell $(i,j)$ with $w_i = w$, $d_j = d$, we could estimate
\begin{equation} \label{e:marginalmodel}
\gamma_{i,j}:=\frac{w d}{N};.
\end{equation}
Let $r_w$ count rows of weight $w$, and let $c_d$ count columns of degree $d$.
Recall that $N_{w,d}$ counts the edges in the bipartite graph $G$ whose left endpoint 
has degree $w$, right endpoint degree $d$, as in (\ref{e:jointdegmat}).
Proposition \ref{p:degreevsweight} is an elementary consequence of the definitions
(\ref{e:logodds}), (\ref{e:normalizellf}), (\ref{e:edgellf}).

\begin{proposition}\label{p:degreevsweight}
The likelihood ratio statistic $\xi$ in (\ref{e:edgellf}) can be written as
\begin{equation} \label{e:llfgrouped}
\xi = \sum_{w, d} \frac{\mu_{w,d} \hat{N}_{w,d}}{\Omega} ;
\quad
\hat{N}_{w,d} := N_{w,d} - \frac{w r_w d c_d} {N}
\end{equation}
where
\[
\mu_{w,d}:=\log{\frac{N - w d}{w d} }; \quad
\Omega^2:=\sum_{w, d} r_w c_d \frac{w d}{N} \left( 1 - \frac{w d}{N} \right) \mu_{w,d}^2.
\]
\end{proposition} 

\textbf{Application: } Consider the null hypothesis $H_0$ that $(Z_{i,j})$ is a Bernoulli matrix
whose parameters follows the model (\ref{e:marginalmodel}), versus the alternative
$H_1$ that entries in cells with higher row total tend to be found in cells with
higher column total. The graph theoretic interpretation of $H_1$ is that, for a typical
edge, left degree and right degree are positively associated.

Under $H_0$ the likelihood ratio statistic (\ref{e:llfgrouped}) is approximately normal$(0,1)$
by Proposition \ref{p:clt-bmm}, but under $H_1$ 
the quantity $\hat{N}_{w,d}$ tends to be positive when $w,d$ are large, which carries lower weight
$\mu_{w,d}$, but negative when $w,d$ are small, which carries higher weight.
In any case, dependency of left degree and right degree will produce bias.
 Hence a two-sided test of size $p$ 
rejects $H_0$ in favor of $H_1$ if $|\xi|$ lies outside the
range of the $p/2$ and $1 - p/2$ quantiles of the standard
normal distribution.

In summary the likelihood ratio statistic fails to detect block structure, but is capable
of testing association between right and left vertex degrees in a sparse bipartite graph.

\textbf{Example: } Generate an incidence matrix whose rows are indexed by
integers generated uniformly at random in
some large window $[1, e^t]$, and whose columns are indexed
by rational primes exceeding $e^{t/4}$. Incidence $Z_{i,j} = 1$ if the $j$-th prime
divides the $i$-th integer. Discard empty rows.
Only 1, 2 or 3 factors in the range $[e^{t/4}, e^t]$ are possible. 
If an integer has 3 factors, the largest is less than
$e^{t/2}$. Hence there is an association
 between left vertex degree (row) and right vertex degree (column).

\subsection{Estimating number of repeated edges} \label{s:duplicates}
In this section, we shall estimate the number of repeated edges for a sequence of
bipartite multigraphs, constructed according to (\ref{e:halfedgecon}).

In the notation of (\ref{e:degreesums}), let $W$ and $D$ denote 
the degrees of a left vertex and a right vertex, respectively, 
selected uniformly at random. Their first and second moments are:
\begin{equation} \label{e:degreemoments}
\E[W] =\frac{N}{n}; \quad
\E[W^2] =\sum_{i=1}^n \frac{w_i^2}{n}; \quad
\E[D] =\frac{N}{m}; \quad
\E[D^2] = \sum_{j=1}^m \frac{d_j^2}{m}.
\end{equation}
Denote by
\begin{equation} \label{e:incidencerate}
\gamma_{i,j}:=\frac{w_i d_j}{N}
\end{equation}
the expected number of instances of the edge $(u_i, v_j)$ in the model (\ref{e:halfedgecon}).
We may associate with any pair of count vectors $\mathbf{w, d}$ as in
(\ref{e:degreesums}) the left (resp. right) degree 
\textbf{coefficients of variation}:
\[
C(\mathbf{w}):=\frac{\E[W^2]}{\E[W]}; \quad 
C(\mathbf{d}):=\frac{\E[D^2]}{\E[D]},
\]
and the \textbf{maximum incidence rate} 
\[
\gamma_*(\mathbf{w, d}):=\max_{i \leq n, j \leq m}{\gamma_{i,j}}.
\]
\begin{proposition}\label{p:duplicates}
Construct, according to (\ref{e:halfedgecon}), a sequence $G_k:=(L_k \cup R_k, E_k)$ of random bipartite (multi)graphs, where $|L_k| \to \infty$, $|R_k| \to \infty$, 
$|E_k| \to \infty$, based upon degree sequences (\ref{e:degreesums})
whose left and right degree coefficients of variation converge
to $C_L \in [0, \infty)$ and $C_R \in [0, \infty)$, respectively, 
and whose maximum incidence rates converge to zero.
Then the expected number of duplicate edges converges to
\begin{equation} \label{e:duplicates}
\frac{C_L C_R} {2},
\end{equation}
while the expected number of edges with three or more instances
converges to zero.
\end{proposition}

\textbf{Remark: } The Poisson approximation technique used in the
proof could no doubt be extended to show that the variance
of the number of duplicate edges also converges to (\ref{e:duplicates}).

\begin{proof} 
If the number $Z_{i,j}$ of edges $(u_i, v_j)$ were Poisson$(\gamma_{i,j})$,
then
\[
\P[Z_{i,j} \geq 2] =
1 - e^{-\gamma_{i,j}} (1 + \gamma_{i,j}) 
= \frac{\gamma_{i,j}^2}{2} + O(\gamma_{i,j}^3)
\]
and $\P[Z_{i,j} = 2]$ admits the same approximation.
These approximations also hold for the multinomial, as we have here,
when the maximum incidence rate converges to zero.
Let 
\[
\xi:=\sum_{i,j} 1_{ \{ Z_{i,j} \geq 2 \} }
\] denote the number of edge positions which are occupied twice or more.
The first moment $\E[\xi]$ is approximated to second order in $(\gamma_{i,j})$ by
\[
\frac{1}{2} \sum_{i,j} \gamma_{i,j}^2
= \frac{1}{2 N^2} \sum_{i=1}^n w_i^2 \sum_{j=1}^m d_j^2
= \frac{n m}{2 N^2} \E[W^2] \E[D^2]
= \frac{1}{2} \frac{\E[W^2] \E[D^2]}{\E[W] \E[D]}.
\]
in the notation of (\ref{e:degreemoments}),
which converges to (\ref{e:duplicates}).
 As for the error term,
\[
 \sum_{i,j} \gamma_{i,j}^3 \leq
\gamma_* \sum_{i,j} \gamma_{i,j}^2
\]
which converges to zero. 
This shows that the number of edge positions which are occupied twice or more
converges to (\ref{e:duplicates}), while
 the number of edges appearing
three or more times converges to zero.
\end{proof}


\textbf{Acknowledgments: } The authors thank Michael Capalbo, John Conroy, Joseph McCloskey, Richard Lehouq, and Karl Rohe
for helpful insights into the literature, and insightful comments.



\end{document}